\documentclass[a4paper, 11pt, oneside]{memoir}

\usepackage[utf8]{inputenc}
\usepackage[T1]{fontenc}
\usepackage{amssymb}
\usepackage[main=UKenglish, ngerman, polish]{babel}
\usepackage{amsmath}
\usepackage{amsthm}
\usepackage{mathrsfs}
\usepackage[shortlabels]{enumitem}
\usepackage{tikz}
\usetikzlibrary{babel, positioning}
\usepackage{flafter}
\usepackage[unicode=true, pdfusetitle]{hyperref} 
\usepackage{memhfixc} 

\newcounter{desclist} 

\setulmarginsandblock{\uppermargin}{\lowermargin}{*}
\setheaderspaces{*}{\headsep}{*} 
\checkandfixthelayout[nearest] 

\nonzeroparskip 
\predate{} 
\date{}
\postdate{}

\thanksmarkseries{arabic} 
\abstractrunin 
\abslabeldelim{.\enskip} 
\setsecnumdepth{subsection} 
\counterwithout{section}{chapter} 
\counterwithout{figure}{chapter} 
\setparaheadstyle{\bfseries\addperiod} 

\hypersetup{hidelinks} 
\hypersetup{bookmarksnumbered} 
\hypersetup{bookmarksdepth = 2} 

\setfloatlocations{figure}{!ht} 
\captionnamefont{\bfseries}
\captiontitlefont{\itshape}
\hangcaption

\theoremstyle{plain} 
\newtheorem{thm}{Theorem}[section] 
\newtheorem{prop}[thm]{Proposition} 
\newtheorem{lem}[thm]{Lemma}
\newtheorem{cor}[thm]{Corollary}

\theoremstyle{definition} 

\newtheorem{exmpl}[thm]{Example}
\newtheorem{rem}[thm]{Remark}

\newlist{tightitemize}{itemize}{2} 
\newlist{tightenumerate}{enumerate}{2}
\newlist{tightdescription}{description}{2}
\setlist[tightitemize]{label=\textbullet, nosep, leftmargin=*, align=left, itemindent=0pt, labelindent=1.5em, labelsep*=0.3em}
\setlist[tightenumerate]{label=(\roman*), nosep, leftmargin=*, align=left, itemindent=0pt, labelindent=1.5em, labelsep*=0.3em}
\setlist[tightdescription]{nosep, leftmargin=!, align=left, itemindent=0pt, labelindent=1.5em, labelsep*=0.3em, labelwidth=4em}

\newcommand*{\NN}{\mathbb{N}}

\newcommand*{\ZZ}{\mathbb{Z}}

\newcommand*{\RR}{\mathbb{R}}

\newcommand*{\Card}[1]{\lvert #1 \rvert} 
\newcommand*{\Close}[1]{\overline{#1}} 
\newcommand*{\DefAs}{:=}
\newcommand*{\AsDef}{=:}
\DeclareMathOperator{\Metric}{d} 

\newcommand*{\Alphab}{\mathscr{A}} 
\DeclareMathOperator{\Per}{Per}  
\DeclareMathOperator{\Aper}{Aper}  
\newcommand*{\Orb}{\mathscr{O}} 

\newcommand*{\Shift}{\sigma} 
\newcommand*{\Rot}{\varrho} 
\newcommand*{\FactOdo}{\pi} 
\newcommand*{\Embed}{\eta} 

\newcommand*{\toeplw}{x} 
\newcommand*{\InfWord}{y} 
\newcommand*{\infWord}{z} 
\newcommand*{\Odo}{\Omega} 
\newcommand{\cplx}[1]{\mathscr{C}_{#1}} 

\title{Almost automorphic subshifts with finiteness conditions for the boundary of the separating cover}
\author{Daniel Sell\thanks{Faculty of Mathematics and Computer Science, Nicolaus Copernicus University in \foreignlanguage{polish}{Toruń, ulica Chopina~12/18, 87-100 Toruń}, Poland. \begin{texttt}dsell@mat.umk.pl\end{texttt}}\and Franziska Sieron\thanks{Faculty of Mathematics and Computer Science, Friedrich Schiller University \foreignlanguage{ngerman}{Jena, Ernst-Abbe-Platz~2, 07743 Jena}, Germany. \begin{texttt}franziska.sieron@uni-jena.de\end{texttt}}}

\begin{document}
\allowdisplaybreaks 

\maketitle

\begin{abstract}
In this article we study orbits of proximal pairs in almost automorphic subshifts. The corresponding orbits in the maximal equicontinuous factor are precisely those orbits that intersect the boundary of the subshift's separating cover. We impose certain finiteness conditions on this boundary and explore the resulting consequences for the subshift, for instance in terms of complexity or the relations between proximal and asymptotic pairs. The last part of our article deals with Toeplitz subshifts without a finite boundary. There we treat the question of necessary conditions and sufficient conditions for the existence of a factor subshift with a finite boundary. Throughout the whole article, we provide explicit constructions for Toeplitz subshifts to illustrate  our findings and the necessity of our assumptions.
\end{abstract}

\section{Introduction}

Asymptotic and proximal pairs play an important role in many areas of symbolic dynamics. In fact, the existence of asymptotic pairs in a subshift distinguishes the interesting non-periodic case from the trivial periodic one, more precisely: a subshift over a finite alphabet contains a non-periodic element, if and only if it contains a non-trivial asymptotic pair, that is, two distinct elements that agree on a half-line (see for instance \cite[pp.\@~18--19]{Ausl_MiniFlowsExtension}). Similarly, elements are called a proximal pair if they agree on arbitrarily large patches, but possibly with ``interruptions'' where they differ. Often it is possible to deduce properties of the subshift from properties of its asymptotic pairs. In this context, bounds on the number of asymptotic components turned out to be especially useful (asymptotic pairs belong to the same component if they differ only by a finite shift). For instance, it was shown in \cite[Theorem~3.1]{DonoDuraMaaPeti-AutoGrLowCompl} that, if the number of asymptotic components in a minimal subshift is finite, then this number is a bound for the cardinality of \( \operatorname{Aut}( X , \Shift ) / \langle \Shift \rangle \), where \( \operatorname{Aut} \) is the automorphism group of the subshift and \( \Shift \) denotes the shift. The finite number of asymptotic component in Sturmian subshifts and simple Toeplitz subshifts can also be used for proving  uniformity of locally constant \(\operatorname{SL}( 2, \RR ) \)-cocycles, since representatives of the components can serve as leading sequences (\cite{GLNS_LeadingSeq_Journ}). Conversely, finiteness of the number of asymptotic components can be deduced from linear complexity along a subsequence (\cite[Lemma~3.2]{DonoDuraMaaPeti-AutoGrLowCompl}) or from good control of a  combinatorial decomposition structure of the words (\cite[Theorem~1.2]{EspinMaass_AutomGrSadic}).

In this article we study proximal and asymptotic orbits of almost automorphic subshifts. There elements are proximal if they have the same image under the factor map to the maximal equicontinuous factor (\cite{Paul_AlmAutomFlows}, \cite{Markley_CharSeq}). We will study proximal orbits via the subshift's separating cover, where they correspond to orbits that intersect the boundary \cite[Proposition~1.1]{Markley_CharSeq}, and via semicocycles, where they correspond to orbits with discontinuities (see Section~\ref{subsec:SepCover} for details). Similar to the above-mentioned results for the asymptotic case, this article is concerned with bounds on the number of proximal components and related notions. We express them as restrictions on the boundary of the subshift's separating cover. This topic is already present in the Markley's discussion of characteristic sequences (that is, almost automorphic points), where he suggests that sequences generated by a cover with certain finiteness properties (so-called Hedlund sequences), ``seem to be a natural class of sequences which we should be able to understand more completely than characteristic sequences in general'' (\cite[Section~3]{Markley_CharSeq}).

Some of our results (and all our examples) concern a special class of almost automorphic subshifts, namely so-called Toeplitz subshifts. Their maximal equicontinuous factor is an odometer, and in fact every extension of an odometer is Borel\(^{*}\) isomorphic (see \cite{DownaLacr_Alm1to1} for the definition) to a Toeplitz subshift (\cite[Theorems~15.2]{Downa_OdomToepl}, \cite[Theorem~7]{DownaLacr_Alm1to1}). Moreover, every positive number \( h > 0 \) can be realized as the entropy of a strictly ergodic Toeplitz subshift with arbitrarily prescribed odometer as maximal equicontinuous factor (\cite[Theorem~B\(^{\prime}\)]{Drwlo_StrictErgodToeplArbitEntro}, but see also \cite[Section~5]{{Wil_ToepNotUniqErgod}} or \cite[Corollary~3]{DownIwan_QuasiUnifConv} for earlier results on realising entropies in Toeplitz systems). For additional background information on Toeplitz subshifts and the rich source of examples that this class of systems has proven to be, we refer the reader to \cite{Downa_OdomToepl} and the references therein.

The proximal elements of a Toeplitz subshift are precisely the elements without the Toeplitz property (sometimes called Toeplitz orbitals). Boundary points of the odometers separating cover translate to non-periodic positions in the orbital. Among the subshifts with finiteness conditions on the boundary, we therefore find Toeplitz subshifts with ``few'' non-periodic positions. Due to their relatively simple structure, they have been widely studied and include for example Toeplitz words with a single hole per period (see for instance \cite{GKBYM_MaxPatternToepl} and \cite{Sell_SimpToepCombETDS} for combinatorial topics, or \cite{LiuQu_Simple} and \cite{LiuQu_Bounded} for Schrödinger operators defined on them), but also Toeplitz subshifts with separated holes (studied for example in \cite{BuKwiat_TopolCentral} from the point of view of automorphism groups). However, other important examples of Toeplitz subshifts do not exhibit such finiteness properties. Notably, this applies to generalised Oxtoby subshifts (see Proposition~\ref{prop:OxtoNotB}), which also have proven to be a rich source of examples and counterexamples: they can for instance define minimal subshifts with an arbitrary prescribed number of ergodic measures (\cite[Section~4]{Wil_ToepNotUniqErgod}), or minimal uniquely ergodic subshifts with positive entropy and trivial centraliser (\cite[Section~2]{BuKwiat_StrictErgod}).

Our article is organised as follows: after a preliminary section on notation and basic definitions, we discuss in Section~\ref{sec:ProxiCompBdry} the connection between proximal orbits and boundary points of a separating cover. There, we also state precisely the finiteness properties that we consider and how they are related to factor subshifts and to our two main example classes, namely Toeplitz subshifts with separated holes and Oxtoby subshifts. In Section~\ref{sec:ConseqOfFin}, we study which properties of the subshift are implied by our finiteness conditions of the boundary. For Toeplitz subshifts which violate the strongest of our conditions of finiteness, we ask in Section~\ref{sec:FactFiniteProxi} if there exists at least a factor subshift which satisfies this condition.

\subsection*{Acknowledgements}

We are indebted to Gabriel Fuhrmann and Daniel Lenz for initiating our work and bringing these questions to our attention. In addition, we thank Aurelia Dymek for interesting discussions and helpful remarks. We gratefully acknowledge the hospitality of the Faculties of Mathematics and Computer Science in Jena and Toruń respectively, which enabled us to conduct part of this work during visits of D.S. to Jena in April 2023 and of F.S. to Toruń in November 2023. The research of D.S. was partially funded by \foreignlanguage{ngerman}{Deutsche Forschungsgemeinschaft} (DFG, German Research Foundation)—\foreignlanguage{ngerman}{Projektnummer} 454053022. The work of F.S. was partially supported by a PhD Scholarship of FSU Jena, funded by the State of Thuringia (``\foreignlanguage{ngerman}{Landesgraduiertenstipendium}'').

\section{Preliminaries}

\subsection{Words, subshifts, factors, proximality}
\label{subsec:Subshifts}

Let \( \Alphab \) be a finite set, called the \emph{alphabet}. The elements of \( \Alphab \) are called \emph{letters} and the elements of \( \Alphab^{\ZZ} \) are known as \emph{(infinite) words}. For \( x \in \Alphab^{\ZZ} \), we use \( x( j ) \) to refer to the letter at position \( j \in \ZZ \) in \( x \), and we write \( x[ i , j ] \) for the finite word that occurs at \( i , i + 1 , \hdots , j \) in \( x \) (all our intervals \( [ i , j ] \) should be read as \( [ i , j ] \cap \ZZ \)). For a finite word \( u \), we denote by \( u ^{n} \) the \( n \)-fold repetition of \( u \), with \( u ^{0} \) denoting the empty word. On \( \Alphab^{\ZZ} \) we consider the \emph{(left-)shift} \( \Shift \colon \Alphab^{\ZZ} \to \Alphab^{\ZZ} \), given by \( \Shift( x )( j ) \DefAs x( j+1) \) for all \( j \in \ZZ \). We equip \( \Alphab^{\ZZ} \) with the product topology, that is, two words \( x , y \in \Alphab^{\ZZ} \) are ``close'' if they agree on a ``large'' interval around the origin. A closed and \( \Shift \)-invariant subset \( X \subseteq \Alphab^{\ZZ} \) (together with the shift action) is called a \emph{subshift}. By \( \Orb( x ) \DefAs \{ \Shift^{n}( x ) : n \in \ZZ \} \) we denote the \emph{orbit} of \( x \in \Alphab^{\ZZ} \). A subshift is called \emph{minimal} if every orbit in it is dense (this is equivalent to every forward orbit being dense, and equivalent to every backward orbit being dense). A subshift is called \emph{aperiodic} if it contains no shift-periodic element. We write \( \cplx{x} \colon \NN \to \NN \) for the \emph{complexity} of a word \( x \), that is, \( \cplx{x}( L ) \) denotes the number of words of length \( L \) which appear in \( x \). Of the many notions that describe the complexity's growth rate in more detail (see for example \cite[Section~2.3]{DonoDuraMaaPeti-AutoGrLowCompl}), we recall the following two: the complexity is called
\begin{tightitemize}
\item{non-superlinear, if \( \liminf_{L \to \infty} \frac{ \cplx{x}( L ) }{ L } < \infty \) holds,}
\item{superpolynomial along a subsequence, if $\limsup_{L\to\infty} \frac{ \cplx{x}(L) }{ \lvert q(L) \rvert }=\infty$ holds for every polynomial $q$.}
\end{tightitemize}
Note that in a minimal subshift, every finite word that appears in some element, appears in every element of the subshift, and hence \( \cplx{x} = \cplx{y} \) holds for all \( x , y \in X \).

Given two subshifts \( X \) and \( Y \), a surjective, continuous and shift-commuting map \( \Psi \colon X \to Y \) is called a \emph{factor map}. In this case, \( Y \) is called a \emph{factor} of \( X \), and \( X \) is called an \emph{extension} of \( Y \). By the theorem of Curtis/Lyndon/Hedlund (see for example \cite[Theorem~6.2.9]{LindMarcus_Coding}), every factor map \( \Psi \)  between subshifts is given by a \emph{sliding block code}, that is, there exist \( J \in \NN_{0} \) and \( \psi \colon \Alphab^{[ -J , J ]} \to \Alphab \) such that \( \Psi( x )( j ) = \psi( x[ j-J , j+J ] ) \) holds for all \( x \in X \)  and \( j \in \ZZ \). If a factor map \( \Psi \) is even bijective, then \( \Psi \) is called a \emph{topological conjugacy} and the subshifts are called \emph{topologically conjugated} (by some authors, the terms ``isomorphism'' and ``isomorphic subshifts'' are used instead). A factor \( Y \) of \( X \), which is neither topologically conjugated to \( X \) nor consists of a single point, is called a \emph{proper factor}. When \( \Psi \colon X \to Y \) is not necessarily bijective, but there is a dense subset $Y_1 \subseteq Y$ such that $\Psi^{-1}(y)$ is a singleton for all $y\in Y_1$, then \( \Psi \), \( Y \) and \( X \) are respectively called  an \emph{almost 1-to-1} map, an \emph{almost 1-to-1 factor} and an \emph{almost 1-to-1 extension}. If $Y$ is minimal, then this is equivalent to the existence of a single $y\in Y$ such that $\Psi^{-1}(y)$ is a singleton.

We will also consider group rotations, that is, a group \( G \) with the action \( \Rot \colon G \to G \), \( g^{\prime} \mapsto g^{\prime} + g \) for a fixed \( g \in G \). A topological group is called monothetic with generator \( g \), if the subgroup \( \langle g \rangle \subseteq G \) is dense. In this case, we always consider \( G \) together with the rotation by a generator. A factor map from a subshift to a group rotation is a continuous, surjective map \( \Psi \colon X \to G \) with \( \Psi \circ \Shift = \Rot \circ \Psi \). For every minimal subshift \( X \), there exists a compact, metrizable, monothetic group \( G \) with a generator \( g \), and a factor map \( \FactOdo_{X} \colon ( X , \Shift ) \to ( G , \Rot) \) such that every factor map from \( X \) to any compact, metrizable, monothetic group \( \widetilde{G} \) factors through \( \FactOdo_{X} \) (see for example \cite[Section~1]{Paul_AlmAutomFlows} and \cite[Theorem~1]{EllisGottsch_HomoTransGr}). The group \( G \) is then called the \emph{maximal equicontinuous factor} of \( X \). Note that the minimality of \( ( X , \Shift ) \) and the properties of the factor map imply the minimality of \( ( G , \Rot) \). 

A minimal subshift \( X \) that is an almost 1-to-1 extension of its maximal equicontinuous factor, is called an \emph{almost automorphic subshift}. The points \( x_{0} \in X \) with \( \FactOdo_{X}^{-1}( \FactOdo_{X}( x_{0} ) ) = \{ x_{0} \} \) are called \emph{almost automorphic points}. In almost automorphic subshifts, aperiodicity is equivalent to an infinite maximal equicontinuous factor. (Indeed, if \( G \) is finite, then it has to be discrete, since it is metrizable. In particular, \(\FactOdo_{X}\) is 1-to-1 on every element of G, and hence X is finite.) A factor subshift \( Y \) of an almost automorphic subshift \( X \) is again almost automorphic, see \cite[Theorem~3.2]{DownDur_FactToepl} or \cite[Proposition~9.9, and Theorem~9.13]{Fursten_RecurrETCombNT} (note also \cite[Proposition~9.14]{Fursten_RecurrETCombNT}, which relates Furstenberg's definition of almost automorphic points to ours). Moreover, in this situation the maximal equicontinuous factor of \( Y \) is a factor of the maximal equicontinuous factor of \( X \).

In this article, we will be especially interested in situations where an almost automorphic subshift \( X \) and its factor subshift \( Y \) have the same maximal equicontinuous factor \( \Odo \). In this case there exists a rotation \( \gamma \colon ( \Odo , \Rot ) \to ( \Odo , \Rot ) \) with \( \FactOdo_{X} = \gamma \circ \FactOdo_{Y} \circ \Psi \), see for example \cite[Section~2]{DownKwiatLac_CritIsomAppl}. Since \( \FactOdo_{X} \) is an almost 1-to-1 map, so is \( \FactOdo_{Y} \circ \Psi \colon X \to \Odo\). In particular, also \( \Psi \) is an almost 1-to-1 map, or in other words: when \( X \) and its factor subshift \( Y \) have the same maximal equicontinuous factor, then \( X \) is an almost 1-to-1 extension of \( Y \). Moreover, in this case the almost automorphic points of \( X \) provide good control over the almost automorphic points of \( Y \); see Proposition~\ref{prop:BndryFactor} for details. Finally, we also note that the aperiodicity of \( X \) then implies that \( Y \) is aperiodic as well, since they have the same infinite maximal equicontinuous factor.

Let now \( \Metric \) denote a metric on \( \Alphab^{\ZZ} \) that is compatible with the topology. Two words \( x_{1} , x_{2} \in \Alphab^{\ZZ} \) are called a \emph{proximal pair} if they satisfy
\[
\liminf_{ n \to -\infty } \Metric( \Shift^{n}( x_{1} ) , \Shift^{n}( x_{2} ) ) = 0 .
\]
They are called an \emph{asymptotic pair} if they even satisfy the stronger condition
\[
\lim_{ n \to -\infty } \Metric( \Shift^{n}( x_{1} ) , \Shift^{n}( x_{2} ) ) = 0  ,
\]
or equivalently: if \( x_{1} \) and \( x_{2} \) are equal on a half-line towards minus infinity. A pair that is proximal but not asymptotic is called a \emph{Li-Yorke pair}. Subshifts without Li-Yorke pairs are called \emph{almost distal}, since they share various properties of distal systems, see \cite[Section~3]{BGKM_LiYorke}. Note that notations in the literature vary -- while we consider negatively proximal and asymptotic pairs, there are also the notions of positively and of two-sided pairs (in which \( \lim_{ n \to -\infty } \) is replaced by \( \lim_{ n \to \infty } \) and \( \lim_{ \lvert n \rvert \to \infty } \), respectively) and of pairs that are proximal in at least one direction. We will discuss this briefly in connection with Proposition~\ref{prop:ProxiBndry} and in Section~\ref{subsec:Li-Yorke}. For related notions such as mean proximality (where the limit of the average distance in \( [ -n , n ] \) is considered), and for their relation to entropy, see for example \cite{DownLacr_ForwMProxi}.

Clearly, whenever \( x_{1} , x_{2} \) are a proximal or asymptotic pair, then so is every finite shift \( \Shift^{n}( x_{1} ) , \Shift^{n}( x_{2} ) \). Therefore, we will often consider proximal and asymptotic relations between orbits: similar to the asymptotic case in \cite[Section~3]{DonoDuraMaaPeti-AutoGrLowCompl}, we say that \( \Orb( x_{1} ) \) and \( \Orb( x_{2} ) \) are proximal (respectively asymptotic), if there are \( x_{1}^{\prime} \in \Orb( x_{1} ) \) and \( x_{2}^{\prime} \in \Orb( x_{2} ) \) which are proximal (respectively asymptotic), or equivalently: if there exists \( n \in \ZZ \) such that \( x_{1} , \Shift^{n}( x_{2} ) \) are a proximal (asymptotic) pair. In general, we cannot expect proximality of orbits to be an equivalence relation. For example, a word \( x_{1} \) which contains arbitrarily long sequences of \( a \)'s and of \( b \)'s, is proximal to the constant sequences \( x_{2} \DefAs \hdots a a a \hdots \) and \( x_{3} \DefAs \hdots b b b \hdots \), but \( \Orb( x_{2} ) = \{ x_{2} \} \) and \( \Orb( x_{3} ) = \{ x_{3} \} \) are clearly not proximal to each other. However, in almost automorphic subshifts proximality of orbits is indeed an equivalence relation, as we will see in Corollary~\ref{cor:ProxiOrbEquiv}. Still following \cite{DonoDuraMaaPeti-AutoGrLowCompl}, we will then call the non-trivial equivalence classes \emph{proximal components}.

\subsection{Toeplitz subshifts and odometers}
\label{subsec:PrelimToepl}

An infinite word \( x \in \Alphab^{\ZZ} \) is called a \emph{Toeplitz word} or \emph{Toeplitz sequence} if it satisfies
\begin{equation}
\label{eqn:DefToepl}
\forall \; j \in \ZZ \quad \exists \; p \in \NN \quad \forall \; n \in \NN : \;\; x( j ) = x ( j + np ) . 
\end{equation}
We denote its orbit closure by \( X_{x} \DefAs \Close{\{ \Shift^{n}( x ) : n \in \ZZ \}} \) and call it a \emph{Toeplitz subshift}. We remark that, while the definition of Toeplitz words includes periodic words as a special case, we are usually not interested in this case (and will sometimes even exclude it). All Toeplitz subshifts are minimal \cite[Theorem~4]{JacobsKeane_01Toeplitz}. In every non-periodic Toeplitz subshift there are elements without property~\eqref{eqn:DefToepl}, that is, elements which are not Toeplitz words (just recall from the introduction that every non-periodic subshift contains a non-trivial asymptotic pair; such a pair cannot be formed by two Toeplitz words, since those would differ on a set with bounded gaps). Following \cite{BJL_ToeplModelSet}, we call those elements \emph{Toeplitz orbitals}, but the reader should be warned that the definitions in the literature vary. For \( p \in \NN \), we denote by
\[
\Per( p , x , a ) \DefAs \{ j \in \ZZ : x( j + n p ) = a \text{ for all } n \in \ZZ \}
\]
the set of all \( p \)-periodic positions with value \( a \) in \( x \), and we write \( \Per( p , x ) \DefAs \bigcup_{a \in \Alphab} \Per( p , x , a ) \) for the set of all \( p \)-periodic positions in \( x \). A \emph{period structure} of a Toeplitz word (see \cite[Section~2]{Wil_ToepNotUniqErgod}) is a sequence \( (p_{l})_{l\in \NN} \) with
\begin{tightenumerate}
\item{\label{item:PerStruc:Div}\( p_{l} \mid p_{l+1} \)},
\item{\label{defi:PerStruc:EssPeriod}\( \emptyset \neq \Per( p_{l} , x ) \neq \Per( p , x ) \) for all \( 0 < p < p_{l} \)},
\item{\label{item:PerStruc:Z}\( \bigcup_{ l \in \NN } \Per( p_{l} , x ) = \ZZ \)}.
\end{tightenumerate}
When necessary, we will additionally set \( p_{0} \DefAs 1 \). A period structure exists for every Toeplitz word, but note that period structures are not unique (for instance, every subsequence is again a period structure). If the density \( \frac{ \Card{ \Per( p_{l} , x ) \cap [ 0, p_{l}-1 ] } }{ p_{l} } \) of \( p_{l} \)-periodic positions converges to \( 1 \) for \( l \to \infty \), the Toeplitz subshift is called \emph{regular}, otherwise \emph{irregular}. We remark that a number \(p_{l} \in \NN \) with property~\ref{defi:PerStruc:EssPeriod} from above is called an \emph{essential period} of \( x \), and that various different but equivalent definitions of this property exist (see \cite[Section~2.3.1]{DKK-P_Minimality} for an overview). We write \( \Aper( p_{l} , x ) \DefAs \ZZ \setminus \Per( p_{l} , x ) \) for the positions in \( x \) that are not \( p_{l} \)-periodic. These positions are called \emph{\( p_{l} \)-holes}. Because of \( \Per( p_{l} , x ) \subseteq \Per( p_{l+1} , x ) \), they form a decreasing sequence of sets and we denote its limit by \( \Aper( x ) \DefAs \ZZ \setminus \bigcup_{l \in \NN} \Per( p_{l} , x ) \). Notice that \( \Aper( x ) \) is non-empty if and only if \( x \) is a Toeplitz orbital. Finally, we remark that for every \( y \in X_{x} \) and for every \( p_{l} \) in a period structure of \( x \), there exists a unique integer \( k = k( y , p_{l} ) \in [0, p_{l}-1] \) such that \( y \) and \( \Shift^{k}( x ) \) agree on their \( p_{l} \)-periodic parts, that is, for every \( a \in \Alphab \)
\begin{equation}
\label{eqn:UniqueShift}
\Per( p_{l} , y, a ) = \Per( p_{l} , \Shift^{k}( x  ) , a)
\end{equation}
holds (see for example \cite[Proposition~3.5]{Drwlo_StrictErgodToeplArbitEntro}, which is based on \cite[Proposition~7]{CortezPetite_GOdo}, or see \cite[Section~8]{Downa_OdomToepl}).

A commonly used technique to construct Toeplitz words is by successive \emph{hole-filling}. For this, we extend the alphabet by an additional symbol ``\( ? \)'' which represents a ``hole'', that is, a position that is not yet filled. We start with a sequence \( ( w_{n} )_{n \geq 1} \) of finite words with holes. We extend each \( w_{n} \) periodically to an infinite word \( w_{n}^{\infty} \in ( \Alphab \cup \{ ? \} )^{\ZZ} \). By \( w_{1}^{\infty} \triangleleft w_{2}^{\infty} \in (\Alphab \cup \{ ? \}) ^{\ZZ} \) we denote the infinite word where we insert \( w_{2}^{\infty} \) letter by letter into the \( ? \)-positions of \( w_{1}^{\infty} \). For instance, with \( w_{1} \DefAs  a??b \) and \( w_{2} \DefAs aa?a?bbb\) (see Example~\ref{ex:UnbdHolesOneBdry} below) we obtain
\begin{align*}
 w_{1}^{\infty} \triangleleft w_{2}^{\infty} = ( & \hdots \, a \, ? \, ? \, b \,  a \, ?\, ?\, b \, a \, ? \, ? \, b \, a \, ? \, ? \, b \, \hdots ) \triangleleft ( \hdots \, a \, a \, ? \, a \, ? \, b \, b \, b \, \hdots )\\
 =\hphantom{(} &\hdots \, a \, a \, a \, b \, a \, ? \, a \, b \, a \, ? \, b \, b \, a \, b \, b \, b \, \hdots \; .
 \end{align*}
Similarly, \( w_{1}^{\infty} \triangleleft w_{2}^{\infty} \triangleleft w_{3}^{\infty} \in (\Alphab \cup \{ ? \}) ^{\ZZ} \) denotes the element where we insert \( w_{3}^{\infty} \) letter by letter into the \( ? \)-positions of \( w_{1}^{\infty} \triangleleft w_{2}^{\infty} \), and so on. Since inserting a word that consists only of \( ? \)'s has no effect, we will assume that every \( w_{n} \) contains at least one letter from \( \Alphab \). Moreover, we choose the position of \( w_{n+1} \) that is filled into the first non-negative hole of \( w_{1}^{\infty} \triangleleft \hdots \triangleleft w_{n}^{\infty} \) in such a way, that all \( ? \)-positions around the origin of \( w_{1}^{\infty} \triangleleft \hdots \triangleleft w_{n}^{\infty} \) are successively filled. In this case, \( w_{1}^{\infty} \triangleleft \hdots \triangleleft w_{n}^{\infty} \) converges to a Toeplitz sequence. In our examples, we always use words \( w_{n} \) whose first and last letters are from \( \Alphab \), and we always fill the first letter of \( w_{n+1} \) into the first non-negative hole of \( w_{1}^{\infty} \triangleleft \hdots \triangleleft w_{n}^{\infty} \).

\paragraph{Toeplitz sequences with separated holes}
If the minimal distance between any two \( p_{l}\)-holes tends to infinity for \( l \to \infty \), we say that \( x \) has \emph{separated holes}. This notion was introduced in \cite{BuKwiat_TopolCentral} and covers interesting classes such as so-called simple Toeplitz subshifts (see for example \cite[Definition~1]{KamZamb_MaxPattCompl}) or, more general, Toeplitz subshifts with a single hole per period, that is, where \( \Aper( p_{l} , x ) \cap [ 0 , p_{l}-1 ] \) is a singleton. As the density of \( \Aper( p_{l} , x ) \) is at most one divided by the minimal distance of holes, it is clear that Toeplitz subshifts with separated holes are regular.

\paragraph{Generalised Oxtoby sequences}
In Toeplitz subshifts with separated holes, the holes get ``more and more isolated''. In contrast, generalised Oxtoby sequences exhibit ``persistent clusters'' of holes: we say that a Toeplitz sequence \( x \) is a \emph{generalised Oxtoby sequence} with respect to a period structure \( (p_{l}) \) of \( x \), if in every interval \( [k p_{l} , ( k+1 ) p_{l} - 1] \), with \( k \in \ZZ \), either all \(p_{l}\)-holes are filled \( p_{l+1} \)-periodically or none of them are, and there are at least two intervals per period that are not filled. More formally:
\begin{tightenumerate}
\item{\label{item:Oxto:All}for every \( l \in \NN \) and every \( k \in [ 0 , \frac{ p_{l+1} }{ p_{l} } - 1 ] \), the set \( \Aper( p_{l+1} , x ) \cap [k p_{l} , ( k+1 ) p_{l} - 1] \) is either empty or equal to \( \Aper( p_{l} , x ) \cap [k p_{l} , ( k+1 ) p_{l} - 1] \),}
\item{and for every \( l \in \NN \), there are at least two \( k \in [ 0 , \frac{ p_{l+1} }{ p_{l} } - 1 ] \) such that the set \( \Aper( p_{l+1} , x ) \cap [k p_{l} , ( k+1 ) p_{l} - 1] \) is non-empty.}
\end{tightenumerate}
Note that a Toeplitz sequence which is Oxtoby for one period structure need in general not be Oxtoby for other period structures (although there are examples which are Oxtoby sequences for all of their period structures). Note also that by the second condition above, Oxtoby sequences cannot be periodic. Moreover, since at least two \( p_{l-1} \)-blocks are not filled \( p_{l} \)-periodically, we have \( \Card{ \Aper( p_{l} , x ) \cap [k p_{l} , ( k+1 ) p_{l} - 1] } \geq 2^{l} \) for every \( l \in \NN\) and every \( k \in \ZZ \). In addition, generalised Oxtoby sequences never have separated holes, since each set \( \Aper( p_{l} , x ) \) contains an interval \( [k p_{1} , ( k+1 ) p_{1} - 1]  \) that has not been completely filled. This interval has at least two holes in it, which are therefore separated by less then \( p_{1} \) (we will prove a stronger version of this statement in Proposition~\ref{prop:OxtoNotB}). Oxtoby sequences, generalising an example of Oxtoby from \cite[Section~10]{Ox_ErgSets}, were originally introduced in \cite{Wil_ToepNotUniqErgod} with slightly different requirements; for details see Example~\ref{exmpl:WilliamsConstr} below. The generalised form presented here appeared under the name ``condition~(*)'' in \cite[Section~1]{BuKwiat_StrictErgod} and as ``generalized Oxtoby sequence'' in \cite[Definition~2]{DownKwiatLac_CritIsomAppl}.

\paragraph{Maximal equicontinuous factor}
Let \( x \) be a Toeplitz word and let \( (p_{l}) \) be a period structure of it. The \emph{odometer associated to \( x \)} is the inverse limit \( \Omega = \varprojlim \ZZ / p_{l}\ZZ \), that is, the set of all sequences \( \omega = ( \omega( 1 ) , \omega( 2 ) , \omega( 3 ) , \hdots ) \in \prod_{l=1}^{\infty} \ZZ / p_{l}\ZZ \) with \( \omega( l+1 ) \equiv \omega( l ) \mod p_{l} \). For \( \omega \in \Odo \) and \( l \in \NN \), the set \( [ \omega ]_{ l } \DefAs \{ \widetilde{\omega} \in \Omega : \widetilde{\omega}[ 1 , l  ] = \omega[ 1 , l ] \} \) is called a \emph{cylinder set}. By \( \Rot \colon \Omega \to \Omega \), \( \omega \mapsto \omega + ( 1 , 1 , 1 , \hdots ) \) we denote the rotation by \( ( 1 , 1 , 1 , \hdots ) \) on \( \Omega \), and we write \( \Orb( \omega ) \DefAs \{ \Rot^{n}( \omega ) : n \in \ZZ \} \) for the orbit of \( \omega \) under \( \Rot \). An alternative notation that is sometimes used in the literature, is to write the odometer as \( \prod_{l \in \NN} \ZZ / \frac{ p_{l} }{ p_{l-1} } \ZZ \) and consider the rotation by \( (1 , 0 , 0 , 0 , \hdots )\) with carry over. It is known that the odometer associated to \( x \) and \( (p_{l}) \) is an almost 1-to-1 factor, as well as the maximal equicontinuous factor, of the subshift \( X_{x} \), see \cite[Theorem~2.2 and Corollary~2.4]{Wil_ToepNotUniqErgod} or for example \cite[Theorem~7.4 and Section~6]{Downa_OdomToepl}. By~\eqref{eqn:UniqueShift}, for each \( y \in X_{x} \) and each \( p_{l} \), there is a uniquely determined shift by \( k = k( y , p_{l} ) \in [0 , p_{l}-1 ] \) that makes \( y \) and \( \Shift^{k}( x ) \) agree on their \( p_{l} \)-periodic parts. The factor map \( \FactOdo_{x} \colon X_{x} \to \Omega \) is given by \( \FactOdo_{x}( y ) \DefAs ( k( y , p_{1}) , k( y , p_{2}), k( y , p_{3}), \hdots ) \), see \cite[Section~8]{Downa_OdomToepl}. Note that, while a general factor map from a Toeplitz subshift \( X_{x} \) would be denoted as \( \FactOdo_{X_{x}} \), we write \( \FactOdo_{x} \) for the specific map that is defined with respect to the Toeplitz sequence \( x \). Note also that the associated odometer is defined in terms of a period structure, which is not unique. However, it follows from the above that the odometers corresponding to different period structures of a Toeplitz word are all isomorphic to each other, since they are all isomorphic to the maximal equicontinuous factor.

Just as factor subshifts of almost automorphic subshifts are again almost automorphic and have related maximal equicontinuous factors (see Section~\ref{subsec:Subshifts}), analogous results hold for the Toeplitz case. This is summarised in the following statement, which combines parts of \cite[Theorems~1.2, 1.3 and 11.1]{Downa_OdomToepl}. We write \( k_{s}( p )  \) for the largest exponent such that \( p^{k_{s}( p )} \) divides \( s \).

\begin{prop}[\cite{Downa_OdomToepl}]
\label{prop:OdoFactIso}
If \( ( X , \Shift ) \) is a Toeplitz subshift with maximal equicontinuous factor given by the odometer \( ( \Omega , \Rot ) \), then every factor subshift \( ( Y , \Shift ) \) of \( ( X , \Shift ) \) is again a Toeplitz subshift, and its maximal equicontinuous factor is a factor of \( ( \Omega , \Rot ) \). An odometer with scale \( ( s_{m} )_{m} \) is a factor of an odometer with scale \( ( t_{n} )_{n} \) if and only if, for every prime number \( p \), we have \( \lim_{m \to \infty} k_{s_{m}}( p ) \leq \lim_{n \to \infty} k_{t_{n}}( p ) \), where we consider the limits to be equal if they are both infinite. The odometers with scales \( ( s_{m} )_{m} \) and \( ( t_{n} )_{n} \) are isomorphic, if and only if, for every prime number \( p \), we have \( \lim_{m \to \infty} k_{s_{m}}( p ) = \lim_{n \to \infty} k_{t_{n}}( p ) \).
\end{prop}

\subsection{Separating covers and related notions}
\label{subsec:SepCover}

We recall different approaches of how almost automorphic subshifts can be defined by  prescribing the induced behaviour on their maximal equicontinuous factor. We start by outlining how these subshifts are obtained in \cite{Paul_AlmAutomFlows}: given a compact metrizable monothetic group \( G \) with rotation \(\Rot \colon G \to  G\) by a generator, a finite cover of closed sets \( C_{0} , \hdots , C_{m-1} \subseteq G \) is called a \emph{separating cover} if
\begin{tightitemize}
\item{each \( C_{a} \) is regular, that is, equal to the closure of its interior,}
\item{the interiors of \( C_{a} \) and \( C_{b} \) are disjoint for all \( a \neq b \),}
\item{for all distinct \( g_{1} , g_{2} \in G \) there exists \( n \in \ZZ \) such that \( \Rot^{n}( g_{1} ) \) and \( \Rot^{n}( g_{2} ) \) lie in the interiors of distinct \( C_{a} \)'s.}
\end{tightitemize}
The cover's boundary is denoted by \( B \DefAs \bigcup_{a\neq b} ( C_{a} \cap C_{b} ) \). Let now \( h \in G \) be an element whose \(\Rot\)-orbit is disjoint from \( B \) (since each \( C_{a} \cap C_{b} \) is nowhere dense, such an element exists by a Baire category argument, cf.\@ \cite[Proposition~2.4]{Paul_AlmAutomFlows}). We define an infinite word \( x \in \{ 0 , \hdots  , m-1 \}^{ \ZZ} \) by setting \( x( j ) \DefAs a \) for \( j \in \ZZ \) if \( \Rot^{j}( h ) \in C_{a} \). In the following statement we summarise various results from \cite[Sections~1 and~2]{Paul_AlmAutomFlows}, but see \cite[Section~1]{Markley_CharSeq} as well for similar results in the case of \( \Card{ \Alphab } = 2 \).

\begin{prop}[\cite{Paul_AlmAutomFlows}]
\label{prop:PaulAlmAutoMef}
The orbit closure \( X \DefAs \Close{\{ \Shift^{n}( x ) : n \in \ZZ \}} \) is a minimal subshift over the alphabet \( \Alphab = \{ 0 , \hdots  , m-1 \} \), and its maximal equicontinuous factor is isomorphic to \( ( G , \Rot) \). More precisely, for every \( y \in X \) there exists a unique \( h \in G \) such that \( y \) encodes the \( C_{a} \)'s along the orbit of \( h \), that is, such that \( y( j ) = a \) implies \( \Rot^{j}( h ) \in C_{a} \). The map \( \FactOdo_{X} \colon X \to G \) that sends \( y \) to this unique element is a factor map from the subshift to its maximal equicontinuous factor. A point \( y \in X\) satisfies \( \FactOdo_{X}^{-1}( \FactOdo_{X}( y ) ) = \{ y \} \) if and only if \( \Orb( \FactOdo_{X}( y ) ) \cap B = \emptyset \) holds. Moreover, two elements \( y_{1} , y_{2} \) are proximal in at least one direction if and only if \( \FactOdo_{X}( y_{1} ) = \FactOdo_{X}( y_{2} ) \) holds. 
\end{prop}

In fact, the proof of \cite[Proposition~1.2]{Paul_AlmAutomFlows} shows that proximality of \( y_{1}, y_{2} \) in at least one direction is equivalent to \( y_{1} , y_{2} \) being positively proximal, equivalent to \( y_{1} , y_{2} \) being negatively proximal, and hence also equivalent to \( y_{1} , y_{2} \) being two-sided proximal (because minimal systems have dense forward and backward orbits). Moreover, \cite[Theorem~2.6]{Paul_AlmAutomFlows} establishes that every almost automorphic subshift can be obtained from a suitable separating cover: for \( X \subseteq \{ 0 , \hdots , m-1 \}^{\ZZ} \), the sets 
\begin{equation}
\label{eq:DefSepCover}
C_{a} \DefAs \FactOdo_{X}( \{ x \in X : x(0) = a \} ) \subseteq \Odo \quad\quad (a = 0 , \hdots , m-1 )
\end{equation}
form a separating cover of the subshift's maximal equicontinuous factor. The subshift obtained from the \( C_{a}\)'s as described above (that is, as closure of an orbit that does not project to the boundary) is then precisely \( X\). We write \( B_{X} \) for the cover's boundary and note that the definition of the separating cover immediately yields the equivalence
\begin{equation}
\begin{aligned}
\label{eq:BndryValues}
& \omega \in B_{X} \\
\!\! \Longleftrightarrow \;\; & \exists \, x_{1} , x_{2} \in \FactOdo_{X}^{-1}( \omega ) \;\; \exists \, a \neq b \in \{ 0 , \hdots , m-1 \} : \,  x_{1}( 0 ) = a , \, x_{2}( 0 ) = b .
\end{aligned}
\end{equation}

\paragraph{Semicocycles}
In the case of Toeplitz subshifts, the concept of separating covers is sometimes expressed in the language of semicocycles. We briefly sketch this approach below, and refer to \cite[Section~5]{DownDur_FactToepl} and \cite[Section~6]{Downa_OdomToepl} for more information. Let hence \( x \) be a Toeplitz word and let \( \Odo \) denote the associated odometer, that is, the maximal equicontinuous factor of \( X_{x} \). We denote by \( \Embed \colon \ZZ \to \Odo \, , \,\, j \mapsto ( j \bmod p_{1} , j \bmod p_{2}  , j \bmod p_{3}  , \hdots ) = \Rot^{j}( (0 , 0 , 0 , \hdots) ) \) the embedding of the integers into the odometer, and we equip \( \Embed( \ZZ ) \subseteq \Omega \) with the induced topology. A \emph{semicocycle} is a continuous map from \( \Embed( \ZZ ) \) to a compact metric space. Specifically, we denote by \( \tau_{x} \)  the semicocycle \( \tau_{x}  \colon \Embed( \ZZ ) \to \Alphab \), \( \Embed( j ) \mapsto x( j ) \). Note that the continuity of \( \tau_{x} \) follows from property~\ref{item:PerStruc:Z} of the definition of a period structure, since for every \( j \in \ZZ \) the value of \( x \) is constant on \( j + p_{l} \ZZ \) for a sufficiently large \( p_{l} \) (and hence \( \tau_{x} \) is constant on the cylinder set \( [ \Embed( j ) ]_{l} \)). Following \cite{Downa_OdomToepl}, we denote by \( F_{x} \subseteq \Omega \times \Alphab \) the closure of the graph of \( \tau_{x} \). Moreover, we write \( F_{x}( \omega ) \DefAs \{ (\omega , a) \in F_{x} : a \in \Alphab \} \) for the set of points in \( F_{x} \) at \( \omega \in \Omega \). Note that \( F_{x}(\omega) \) is a singleton for every \( \omega \in \Embed( \ZZ ) \) by the continuity of \( \tau_{x} \) on \( \Embed( \ZZ ) \). The sets
\[ C_{a} \DefAs \{ \omega \in \Omega : ( \omega , a ) \in F_{x} \} \, , \;\text{with } a \in \Alphab \!, \]
form a separating cover, whose boundary \( B_{X} = \bigcup_{a\neq b} ( C_{a} \cap C_{b} ) \subseteq \Omega \setminus \Embed( \ZZ )  \) consists precisely of those \( \omega \in \Omega \) for which \( F_{x}(\omega) \) is not a singleton (also called the  \emph{set of discontinuities of \( \tau_{x} \)}). Indeed, it is easy to check the following properties.
\begin{prop}[Properties~(A2) and (A1) in {\cite[Section~8]{Downa_OdomToepl}}]
\label{prop:AperDx}
\leavevmode
\begin{tightenumerate}
\item{\label{item:A2}For \( \omega \in \Odo \) and \( a \in \Alphab \) we have: \( ( \omega , a) \in F_{x} \Longleftrightarrow \exists \, y \in \FactOdo_{x}^{-1}( \omega ) \) with \( y( 0 ) = a \).}
\item{\label{item:A1}For each \( y \in X_{x} \) we have \( \Aper( y ) = \{ n \in \ZZ : \Rot^{n}( \FactOdo_{x}( y ) ) \in B_{X} \} \).}
\end{tightenumerate}
\end{prop}
It follows immediately from Property~\ref{item:A2} and Equation~\eqref{eq:DefSepCover} that the \( C_{a} \)'s form a separating cover. We also note that by Property~\ref{item:A1}, \( B_{X} \) is empty if and only if \( x \) is periodic (see \cite[Lemma~1.4]{Markley_CharSeq} as well).

\paragraph{CPS windows}

Cut-and-project schemes (``CPS'' for short; see \cite[Chapter 7]{BaakeGrimm_Aperio} or \cite{Moody_ModelSurv} for general background information) are another way to describe Toeplitz subshifts: for every binary Toeplitz subshift \( X_{x} \) with associated odometer \( \Omega \), there exists by \cite[Theorem~1]{BJL_ToeplModelSet} a CPS with internal space \( \Omega \), physical space \( \ZZ \) and lattice \( L \DefAs \{ ( j,  \Embed( j ) ) : j \in \ZZ \} \subseteq \ZZ \times \Omega \). While we do not use it in this article, we would like to point out the close relation between the semicocycle \( \tau_{x} \) and its discontinuities on the one hand, and the CPS and the boundary of its windows on the other hand. We recall the definition of a CPS-window from \cite{BJL_ToeplModelSet} (with the straightforward generalisation from a binary to an arbitrary finite alphabet): let \( x \) be a Toeplitz word with period structure \( ( p_{l} ) \). For \( a \in \Alphab \) and \( l \in \NN \), we define \( U_{a,l} \DefAs \{ \omega \in \Omega : \omega( l ) \in \Per( p_{l} , x , a ) \} \). We note that this implies \( U_{a,l} \subseteq U_{a,l+1} \) for all \( a \in \Alphab \), and \( U_{a,l} \cap U_{b,l} = \emptyset \) for all \( a \neq b \). We set \( U_{a} \DefAs \bigcup_{l \in \NN} U_{a,l} \) and define the \( a \)-window as \( W_{a} \DefAs \Close{U_{a}} \). It has the property that the projection of \( (\ZZ \times W_{a} ) \cap L  \) to \( \ZZ \) yields precisely the set \( \{ j \in \ZZ : x(j) = a \} \) of \( a \)-positions in \( x \), see \cite[Theorem~1]{BJL_ToeplModelSet}. In addition, \( W_{a} \cap \Embed( \ZZ ) \) coincides with \( \tau_{x}^{-1}( a ) \). Indeed, for every \( j \in \ZZ \) there exist \(l \in \NN \) and \( a \in \Alphab \) with \( j \in \Per( p_{l} , x , a ) \). This implies \( \tau_{x}( \Embed( j ) ) = a \), but also \( \Embed( j ) \in U_{a,l} \subseteq U_{a} \subseteq W_{a} \). It is easily checked that \( U_{a} \cap U_{b} = \emptyset \) holds for all \( a \neq b \). Moreover, as in the proof of \cite[Theorem~1]{BJL_ToeplModelSet} it can be shown that the combined boundaries of the windows are given by \( \bigcup_{ a \in \Alphab } \partial W_{a} = \bigcup_{a \neq b} ( W_{a} \cap W_{b} ) \). In other words, they consist of all \( \omega \in \Omega \) which can be approximated by sequences in \( \Embed( \ZZ ) \) that lie in at least two different sets \( U_{a}, U_{b} \). As these are precisely the discontinuity points of \( \tau_{x} \), we obtain \( \bigcup_{ a \in \Alphab } \partial W_{a} = B_{X} \).

\section{Proximal components and boundary points}
\label{sec:ProxiCompBdry}

In this section we begin our study of proximal orbits and their relation to the boundary of a separating cover. We introduce three conditions of finiteness for the boundary and put them into context. In subsequent sections we will then investigate the consequences of these conditions, and discuss when factor subshifts satisfy them. We start with the following straightforward but helpful observation, which links the behaviour of elements to  the behaviour of orbits.

\begin{prop}
\label{prop:ElemToOrb}
Let \( \FactOdo_{X} \colon X \to \Odo \) be a factor map from an almost automorphic subshift to its maximal equicontinuous factor. Two elements \( x_{1} ,  x_{2} \in X \) from the same orbit satisfy \( \FactOdo_{X}( x_{1} ) = \FactOdo_{X}( x_{2} ) \) if and only if \( x_{1} =  x_{2} \) holds. In particular, if \( \FactOdo_{X}( x_{1} ) = \FactOdo_{X}( x_{2} ) \) holds and \( x_{1} , x_{2} \) are not equal (respectively not asymptotic), then also \( \Orb( x_{1} )  , \Orb( x_{2} ) \) are not equal (respectively not asymptotic).
\end{prop}

\begin{proof}
Let \( x_{1} ,  x_{2} \) be elements from the same orbit, let \( n \in \ZZ \) be such that \( x_{2} = \Shift^{n}( x_{1} ) \) holds, and assume that we have \( \FactOdo_{X}( x_{1} ) = \FactOdo_{X}( x_{2} ) \). We obtain \( \FactOdo_{X}( x_{1} ) = \FactOdo_{X}( x_{2} ) = \Rot^{n}( \FactOdo_{X}( x_{1} ) ) \). For a periodic subshift, \( \FactOdo_{X}( x_{1} ) = \Rot^{n}( \FactOdo_{X}( x_{1} ) ) \) implies that \( n \) is a multiple of the period. For an aperiodic subshift (and hence an infinite maximal equicontinuous factor), the minimality of \( ( \Odo , \Rot ) \) implies \( n = 0 \). In both cases, we conclude \( x_{1} = x_{2} \). In particular, if \( \FactOdo_{X}( x_{1} ) = \FactOdo_{X}( x_{2} ) \) holds, then  \( x_{1} \neq x_{2} \) implies  \( \Orb( x_{1} ) \neq \Orb( x_{2} ) \). Finally, let us assume that \( x_{1} , x_{2} \) are not asymptotic. As we have seen above, \( \FactOdo_{X}( x_{1} ) = \FactOdo_{X}( x_{2} ) \) implies that there is no \( x_{1}^{\prime} \in \Orb( x_{1} ) \setminus \{ x_{1} \} \) with \( \FactOdo_{X}( x_{1}^{\prime} ) = \FactOdo_{X}( x_{2} ) \). Thus, it follows from Proposition~\ref{prop:PaulAlmAutoMef} that \( x_{1}^{\prime} \) and \( x_{2} \) are not proximal. In particular, \( x_{1}^{\prime} , x_{2} \) are not asymptotic for any \( x_{1}^{\prime} \in \Orb( x_{1} ) \).
\end{proof}

\begin{prop}[{see \cite[Proposition~1.1]{Markley_CharSeq}, \cite[Proposition~1.2]{Paul_AlmAutomFlows}}]
\label{prop:ProxiBndry}
Let \( X \subseteq \Alphab^{\ZZ} \) be an almost automorphic subshift with maximal equicontinuous factor \( \Odo \) and factor map \( \FactOdo_{X} \colon X \to \Odo \). Let \( C_{0} , \hdots , C_{\Card{\Alphab}-1} \) be a separating cover that generates \( X \) and let \( B_{X} \) denote its boundary. Then for every \( x_{1} \in X\) the following are equivalent:
\begin{tightenumerate}
\item{The relation \( \FactOdo_{X}( \Orb( x_{1}) ) \cap B_{X} \neq \emptyset \) holds.}
\item{There exists an orbit \( \Orb( x_{2}) \neq \Orb( x_{1}) \) with \( \FactOdo_{X}( \Orb( x_{2}) ) = \FactOdo_{X}( \Orb( x_{1}) ) \).}
\end{tightenumerate}
Moreover, \( \FactOdo_{X}( \Orb( x_{2}) ) = \FactOdo_{X}( \Orb( x_{1}) ) \) holds if and only if \( \Orb( x_{2}) , \Orb( x_{1}) \) are negatively (equivalently: positively; equivalently: two-sided) proximal.
\end{prop}

\begin{proof}
Essentially this follows from Proposition~\ref{prop:PaulAlmAutoMef}: if \( \FactOdo_{X}( \Orb( x_{1}) ) \cap B_{X} \neq \emptyset \) holds, then \( \FactOdo_{X}^{-1}( \FactOdo_{X}( x_{1} ) ) \) is not a singleton. Hence there exists \( x_{2} \neq x_{1} \) with \( \FactOdo_{X}( x_{2} ) = \FactOdo_{X}( x_{1} ) \), and therefore with \( \FactOdo_{X}( \Orb( x_{2} ) ) = \FactOdo_{X}( \Orb( x_{1} ) ) \). Moreover, Proposition~\ref{prop:ElemToOrb} implies \( \Orb( x_{2} ) \neq \Orb( x_{1} ) \). Conversely, assume that there is an orbit \( \Orb( x_{2} ) \neq \Orb( x_{1} ) \) with \( \FactOdo_{X}( \Orb( x_{2} ) ) = \FactOdo_{X}( \Orb( x_{1} ) ) \), and let \( n \in \ZZ \) be such that \( \FactOdo_{X}( x_{1} ) = \FactOdo_{X}( \Shift^{n}( x_{2} ) ) \) holds. This yields \( \FactOdo_{X}^{-1}( \FactOdo_{X}( x_{1} ) ) \supseteq \{ x_{1} , \Shift^{n}( x_{2} ) \} \). Invoking Proposition~\ref{prop:PaulAlmAutoMef} once more, we obtain \( \Orb( \FactOdo_{X}( x_{1} ) ) \cap B_{X} \neq \emptyset \). Finally, we notice for the ``moreover''-part of our assertion that:
\begin{align*}
&\FactOdo_{X}( \Orb( x_{2}) ) = \FactOdo_{X}( \Orb( x_{1}) )\\
\Longleftrightarrow {} & \exists \, n \in \ZZ \text{ such that } \FactOdo_{X}( \Shift^{n}( x_{2}) ) = \FactOdo_{X}( x_{1} ) \text{ holds}\\
\Longleftrightarrow {} & \exists \,n \in \ZZ \text{ such that } \Shift^{n}( x_{2}) , x_{1} \text{ are proximal (see Proposition~\ref{prop:PaulAlmAutoMef})}\\
\Longleftrightarrow {} & \Orb( x_{2}) ,  \Orb( x_{1}) \text{ are proximal.}
\end{align*}
As noted after Proposition~\ref{prop:PaulAlmAutoMef}, for the proximality notion here we can equivalently consider positive, negative or two-sided proximality.
\end{proof}

\begin{cor}
\label{cor:ProxiOrbEquiv}
In almost automorphic subshifts, proximality of orbits is an equivalence relation.
\end{cor}

\begin{proof}
Clearly, orbit proximality is symmetric and reflexive. The transitivity follows immediately from the last part of Proposition~\ref{prop:ProxiBndry}: if \( \Orb( x_{1} ) , \Orb( x_{2} ) \) and \( \Orb( x_{2} ) , \Orb( x_{3} ) \) are pairs of proximal orbits, then \( \FactOdo_{X}( \Orb( x_{1} ) ) = \FactOdo_{X}( \Orb( x_{2} ) ) \) and \( \FactOdo_{X}( \Orb( x_{2} ) ) = \FactOdo_{X}( \Orb( x_{3} ) ) \) imply  \( \FactOdo_{X}( \Orb( x_{1} ) ) = \FactOdo_{X}( \Orb( x_{3} ) ) \), and hence proximality of \( \Orb( x_{1} ) \) and \( \Orb( x_{3} ) \).
\end{proof}

We note in particular that the preimage \( \FactOdo_{X}^{-1}( \Orb( \omega ) ) \) of an orbit in \( \Omega \) is precisely an equivalence class under the proximality relation. If such an equivalence class consists of more than a single orbit, we call it a \emph{proximal component} (in analogy to the asymptotic case, see \cite[Section~3]{DonoDuraMaaPeti-AutoGrLowCompl}). Thus, by Proposition~\ref{prop:ProxiBndry} the equivalence class of \( \Orb( x ) \) is a proximal component if and only if \( \Orb( \FactOdo_{X}( x ) ) \) intersects \( B_{X} \).

Recall now that every almost automorphic subshift can be defined via a separating cover as given in~\eqref{eq:DefSepCover}. In this article we study the consequences of certain finiteness properties of the cover's boundary \( B_{X} \). We are interested in three different properties:
\begin{tightdescription}
\item[(FPC)]{\renewcommand{\thedesclist}{(FPC)}\refstepcounter{desclist}\label{item:FinOrb}Only finitely many orbits \( \Orb( \omega ) \subseteq \Omega \) intersect \( B_{X} \).}
\item[(HS)]{\renewcommand{\thedesclist}{(HS)}\refstepcounter{desclist}\label{item:FinCap}For every \( \Orb( \omega ) \), the intersection with \( B_{X} \) is finite (possibly empty).}
\item[(FB)]{\renewcommand{\thedesclist}{(FB)}\refstepcounter{desclist}\label{item:FinD}The set \( B_{X} \) is finite.}
\end{tightdescription}

We note that:
\begin{tightitemize}
\item{Property~\ref{item:FinOrb} is equivalent to \textbf{F}initely many \textbf{P}roximal \textbf{C}omponents; for \( 0 \text{\,-} 1 \)-sequences, a separating cover with property~\ref{item:FinCap} is called a \textbf{H}edlund \textbf{S}et in \cite[Section~3]{Markley_CharSeq}; and property~\ref{item:FinD} denotes a \textbf{F}inite \textbf{B}oundary.}
\item{Toeplitz subshift with separated holes always satisfy \ref{item:FinCap}, since separated holes are equivalent to at most one boundary point within each orbit \( \Orb( \omega ) \) by Proposition~\ref{prop:AperDx}~\ref{item:A1}.}
\item{Generalised Oxtoby sequences never satisfy \ref{item:FinCap}, see Proposition~\ref{prop:OxtoNotB} below.}
\item{\ref{item:FinD} holds, if and only if \ref{item:FinOrb} and \ref{item:FinCap} hold.}
\end{tightitemize}

Moreover, each of the properties \ref{item:FinOrb}, \ref{item:FinCap} and \ref{item:FinD} is preserved when going to a factor subshift with the same maximal equicontinuous factor (recall from Section~\ref{subsec:Subshifts} that the factor subshift is again almost automorphic). We obtain this in Corollary~\ref{cor:PresUnderFact} as a consequence of the following observation. We remark that results related to  Proposition~\ref{prop:BndryFactor} can also be found in \cite[Section~2]{Markley_CharSeq} for the case \( \Card{ \Alphab } = 2 \), and in \cite[Section~2]{DownKwiatLac_CritIsomAppl} for the Toeplitz case (for the latter, we note that we can indeed obtain what there is called a ``homomorphism over zero'' by making the correct choice for the factor map to the maximal equicontinuous factor).

\begin{prop}
\label{prop:BndryFactor}
Let \( X \subseteq \Alphab^{\ZZ} \) be an almost automorphic subshift, let \( \Psi \) be a factor map based on a sliding block code \( \psi \colon \Alphab^{[ -J , J ]} \to \Alphab \) and let \( Y \DefAs \Psi( X ) \) be the resulting factor subshift. Assume that \( X \) and \( Y \) have the same odometer \( \Omega \) as their maximal equicontinuous factor. Let \( C_{0} , \hdots , C_{\Card{\Alphab}-1} \) and \( D_{0} , \hdots , D_{\Card{\Alphab}-1} \) denote separating covers which define \( X \) respectively \( Y \) as in~\eqref{eq:DefSepCover}, and let \( B_{X} \) and \( B_{Y} \) denote their boundaries. Then there is a bijection \( \gamma \colon \Odo \to \Odo \) such that for every \( \omega \in B_{Y} \) there exists \( j \in [ -J , J ] \) with \( \Rot^{j}( \gamma( \omega ) ) \in B_{X} \). In particular, if \( \Orb( \omega) \) intersects \( B_{Y} \), then \( \Orb( \gamma( \omega ) ) \) intersects \( B_{X} \).
\end{prop}

\begin{proof}
First we note that \( \FactOdo_{X} \colon X \to \Odo \) and \( \FactOdo_{Y} \circ \Psi \colon X \to \Odo \) both are factor maps from the subshift \( X \) to its maximal equicontinuous factor. Thus, there exists a rotation (and hence bijection) \( \gamma \colon ( \Odo , \Rot ) \to ( \Odo , \Rot ) \) with \( \FactOdo_{X} = \gamma \circ \FactOdo_{Y} \circ \Psi \), see for example \cite[Section~2]{DownKwiatLac_CritIsomAppl}. Consider now \( \omega \in B_{Y} \) and let \( a \neq b \in \{ 0 , \hdots , \Card{\Alphab}-1 \} \) be such that \( \omega \in D_{a} \cap D_{b} \) holds. By definition of the separating cover, see~\eqref{eq:DefSepCover}, there are therefore \( y_{1} , y_{2} \in \FactOdo_{Y}^{-1}( \omega ) \) with \( y_{1}( 0 ) = a \) and \( y_{2}( 0 ) = b \). We now consider  \( x_{1} \in \Psi^{-1}( y_{1} ) \) and \( x_{2} \in \Psi^{-1}( y_{2} ) \). As \( \Psi \) is given by a sliding block code on \( [ -J , J ] \), there exists \( j \in [ -J , J ] \) with \( x_{1}( j ) \neq x_{2}( j ) \). We obtain
\[ C_{ x_{1}( j ) } \ni \FactOdo_{X}( \Shift^{j}( x_{1} ) ) = \Rot^{j}( \gamma ( \FactOdo_{Y} ( \Psi ( x_{1} )))) = \Rot^{j}( \gamma( \omega) ) , \]
and similarly \( C_{ x_{2}( j ) } \ni \FactOdo_{X}( \Shift^{j}( x_{2} ) ) = \Rot^{j}( \gamma( \omega) ) \), which yields \( \Rot^{j}( \gamma( \omega) ) \in B_{X} \).
\end{proof}

\begin{cor}
\label{cor:PresUnderFact}
Under the assumptions of Proposition~\ref{prop:BndryFactor} we have the following:
\begin{tightenumerate}
\item{\label{item:PcFactor}The number of proximal components of \( Y \) is at most the number of proximal components in \( X \). In particular, the number of proximal components is invariant under topological conjugacies.}
\item{\label{item:PropFactor}If \( X \) has any of the properties \ref{item:FinOrb}, \ref{item:FinCap} or \ref{item:FinD}, then also \( Y \) has these properties.}
\end{tightenumerate}
\end{cor}

\begin{proof}
\emph{\ref{item:PcFactor}} By Proposition~\ref{prop:ProxiBndry}, a proximal component of \( X \) is the equivalence class of an orbit whose image in the maximal equicontinuous factor intersects \( B_{X} \). Since \( \gamma \) in Proposition~\ref{prop:BndryFactor} is injective, the number of such orbits in \( Y \) is less or equal to the number of such orbits in \( X \). If \( X \) and \( Y \)  are topologically conjugate, then they are factors of each other, and thus the number of proximal components must be the same.

\emph{\ref{item:PropFactor}} We have already seen in part~\ref{item:PcFactor} that \ref{item:FinOrb} is preserved under factor maps. Moreover, by Proposition~\ref{prop:BndryFactor} we have \( \Card{ \Orb( \omega ) \cap B_{Y} } \leq (2J+1) \cdot \Card{ \Orb( \gamma( \omega ) ) \cap B_{X} } \), which shows that also \ref{item:FinCap} is preserved. Combining both results proves the claim for \ref{item:FinD}.
\end{proof}

\begin{prop}
\label{prop:OxtoNotB}
For every generalised Oxtoby sequence \( x \) and for every \( y \in X_{x} \) we have \( \Card{ \Aper( y ) } \in \{ 0 , \infty \} \). In particular, no generalised Oxtoby subshift satisfies property~\ref{item:FinCap} or property~\ref{item:FinD}.
\end{prop}

\begin{proof}
Let \( (p_{l}) \) denote a period structure with respect to which \( x \) is a generalised Oxtoby sequence. We fix an arbitrary \( t \in \NN \), and assume that \( y \in X_{x} \) is such that \( \Aper( y ) \neq \emptyset \) holds. We will show \( 2^{t} \leq \Card{ \Aper(y) } \), which implies the first part of the assertion. Since \( \Card{ \Aper(y) } \) is \( \Shift \)-invariant, we will assume without loss of generality that \( 0 \in \Aper( y ) \) holds.

Using that \( ( \Aper( p_{l} , y ) )_{l} \) is a decreasing sequence of sets with \( \bigcap_{l=1}^{\infty} \Aper( p_{l} , y ) = \Aper( y ) \), we note that there exists \( L \in \NN \) with
\[ \Aper( p_{L} , y ) \cap [ -p_{t}+1 , p_{t}-1 ] = \Aper( y ) \cap [ -p_{t}+1 , p_{t}-1 ] . \]
We remark that \( L \geq t \) holds, because \( [0 , p_{t}-1] \) contains some \( p_{t-1} \)-holes which are filled \( p_{t} \)-periodically. Recall now that the \( p_{L} \)-periodic parts in \( x \) and \( y \) differ only by a finite shift of \( \FactOdo_{x}(y)(L) \), which yields
\begin{align*}
&\Card{ \Aper( p_{L} , x ) \cap [-p_{t}+1 + \FactOdo_{x}(y)(L) , p_{t}-1 + \FactOdo_{x}(y)(L)] } \\*
={}& \Card{ \Aper( p_{L} , y ) \cap [-p_{t}+1 , p_{t}-1] }\\
={}& \Card{\Aper( y ) \cap [ -p_{t}+1 , p_{t}-1 ] } \\
\leq{}& \Card{ \Aper( y ) }
\end{align*}
Moreover, \( 0 \in \Aper( y ) \subseteq \Aper( p_{L} , y ) \) implies \(\FactOdo_{x}(y)(L) \in \Aper( p_{L} , x ) \). We consider the interval of the form \( [n p_{t} , (n+1)p_{t} -1 ] \) in \( x \) around \( \FactOdo_{x}(y)(L) \). Because this interval contains a \( p_{L}\)-hole of \( x \), by part~\ref{item:Oxto:All} in the definition of Oxtoby sequences all \( p_{t} \)-holes in the interval are \( p_{L} \)-holes. Since \( x \) is a generalised Oxtoby sequence, the interval (like every interval of length \( p_{t} \)) contains at least \( 2^{t} \)-many \( p_{t} \)-holes, see Section~\ref{subsec:PrelimToepl}. We obtain
\[ 2^{t} \leq \Card{ \Aper( p_{L} , x ) \cap [-p_{t}+1+ \FactOdo_{x}(y)(L) , p_{t}-1+ \FactOdo_{x}(y)(L)] } \leq \Card{ \Aper( y ) } .\]

For the second part of our assertion, recall that generalised Oxtoby subshifts are non-periodic by definition. Therefore, there exists at least one Toeplitz orbital \( y \in X_{x} \). By the first part of this proposition and by Proposition~\ref{prop:AperDx}~\ref{item:A1}, this implies \( \infty = \Card{ \Aper( y ) } = \Card{ \Orb( \FactOdo_{x}( y ) ) \cap B_{X} } \), and hence rules out properties~\ref{item:FinCap} and \ref{item:FinD}.
\end{proof}

\section{Consequences of \ref{item:FinOrb}, \ref{item:FinCap} and \ref{item:FinD}}
\label{sec:ConseqOfFin}

\subsection{Li-Yorke pairs}
\label{subsec:Li-Yorke}

As we have seen in Proposition~\ref{prop:ProxiBndry}, the notions of negatively, positively and two-sided proximal orbits are all equivalent in almost automorphic subshifts. We will see in Example~\ref{exmpl:WilliamsConstr} that this need not be the case for asymptotic orbits, and that proximal orbits need not be asymptotic. However, as we show below, all of these equivalences hold for almost automorphic subshifts under the additional assumption of property~\ref{item:FinCap} (which includes for example Toeplitz subshifts with separated holes). Additionally, we refer the reader to \cite[Section~3]{Markley_CharSeq} for more results about separating covers with property~\ref{item:FinCap}, especially in the case where \( \Odo = \RR^{n} / \ZZ^{n} \) is the \( n \)-dimensional torus.

\begin{prop}
\label{prop:BproxiAsympt}
Assume that an almost automorphic subshift \( X \) satisfies \ref{item:FinCap}. For any two orbits \( \Orb( x_{1} ) \) and \( \Orb( x_{2} ) \) in \( X \), all of the following notions are equivalent: negatively proximal, positively proximal, two-sided proximal, negatively asymptotic, positively asymptotic and two-sided asymptotic.
\end{prop}

\begin{proof}
It is clear that two-sided asymptotic implies negatively asymptotic and positively asymptotic, and that each of them implies proximal (all proximality notions are equivalent by Proposition~\ref{prop:ProxiBndry}). Thus, it only remains to show that proximal orbits are  two-sided asymptotic. 

Let hence \( \Orb( x_{1} ) \) and \( \Orb( x_{2} ) \) be proximal orbits. By Proposition~\ref{prop:ProxiBndry} this implies \( \FactOdo_{X}( \Orb( x_{1} ) ) = \FactOdo_{X}( \Orb( x_{2} ) ) \). Since every finite shift of \( x_{2} \) defines the same orbit as \( x_{2} \), we can assume without loss of generality that \( \FactOdo_{X}( x_{1} ) = \FactOdo_{X}( x_{2} ) \) holds. Since the subshift is almost automorphic, it is generated by a separating cover, see Section~\ref{subsec:SepCover}. Let \( B_{X} \) denote the cover's boundary. After a rotation by \(\Rot^{j} \), Equation~\eqref{eq:BndryValues} yields 
\[ \{ j \in \ZZ : x_{1}( j ) \neq x_{2}( j ) \} \subseteq \{ j \in \ZZ : \Rot^{j}( \FactOdo_{X}( x_{1} ) ) \in B_{X} \}, \]
where the right hand side is finite by assumption. Thus, \( x_{1} \) and \( x_{2} \) form a two-sided asymptotic pair, and hence \( \Orb( x_{1} ) \) and \( \Orb( x_{2} ) \) are two-sided asymptotic orbits.
\end{proof}

\begin{cor}
Assume that an almost automorphic subshift \( X \) satisfies \ref{item:FinD}. Then \( X \) has a finite number of asymptotic components. In particular, its automorphism group (that is, the set of all homeomorphisms \(X \to X \) that commute with the shift) is virtually \( \ZZ \).
\end{cor}

\begin{proof}
As noted earlier, \ref{item:FinD} is equivalent to \ref{item:FinOrb} and \ref{item:FinCap}. Hence \( X \) has a finite number of asymptotic components. Moreover it follows from Proposition~\ref{prop:BproxiAsympt} that proximal and asymptotic components are the same. The last part of the assertion is now just an application of \cite[Theorem~3.1]{DonoDuraMaaPeti-AutoGrLowCompl}.
\end{proof}

\begin{cor}
\label{cor:LiYorkeEntroZero}
Assume that an almost automorphic subshift \( X \) satisfies \ref{item:FinD}. Then \( X \) has topological entropy zero.
\end{cor}

\begin{proof}
By Proposition~\ref{prop:BproxiAsympt}, an almost automorphic subshift with property \ref{item:FinD} has no Li-Yorke pairs. Our claim thus follows from \cite[Corollary~2.4]{BGKM_LiYorke}, which asserts that positive entropy implies that there are uncountably many elements which pairwise form Li-Yorke pairs. Note that in \cite{BGKM_LiYorke}, positively proximal and positively asymptotic pairs are considered, but in the setting of Proposition~\ref{prop:BproxiAsympt} they are equivalent to negatively or two-sided pairs.
\end{proof}

In absence of property~\ref{item:FinCap}, proximal and asymptotic orbits may or may not be the same. In fact, both types of behaviour can occur within the class of generalised Oxtoby subshifts (which never satisfy~\ref{item:FinCap}, see Proposition~\ref{prop:OxtoNotB}): in Example~\ref{exmpl:OxtProxiAsym} we will encounter an Oxtoby subshift where every proximal orbit is asymptotic; in particular, it follows that condition~\ref{item:FinCap} in Proposition~\ref{prop:BproxiAsympt} is not necessary. However, in Proposition~\ref{prop:FiniteFiber} we will show that ``non-\ref{item:FinCap} plus certain conditions'' is enough to ensure that the equivalence of proximal and asymptotic orbits fails, and in Example~\ref{exmpl:WilliamsConstr} we show that there are generalised Oxtoby subshifts to which Proposition~\ref{prop:FiniteFiber} applies.

\begin{exmpl}
\label{exmpl:OxtProxiAsym}
Via a hole-filling procedure (see Section~\ref{subsec:PrelimToepl}) we construct a generalised Oxtoby sequence in which every proximal pair of orbits is also asymptotic. Firstly, for \( l \in \NN \) and \( i \in [ 1 , 2^{l} ] \) we define \( u^{(l)}_{i} \in \{ a , b \}^{[1,2^{l}]} \) to be the word of length \( 2^{l} \) that has a single \( b \) at position \( i \), and value \( a \) at all other positions. Now we set
\begin{align*}
w_{1} \,={} & ba \, ? ? \, ? ? \, ab = u^{(1)}_{1} \; ?^{4} \; u^{(1)}_{2} ,\\
w_{2} \,={} & baaa \, abaa \, ???? \, ???? \, aaba \, aaab = u^{(2)}_{1} \; u^{(2)}_{2} \; ?^{8} \; u^{(2)}_{3} \; u^{(2)}_{4}\\
\text{and in general} \;\; w_{l} \DefAs {} & u^{(l)}_{1} \hdots \, u^{(l)}_{2^{l-1}} \; ?^{2^{l+1}} \; u^{(l)}_{2^{l-1}+1} \hdots \, u^{(l)}_{2^{l}} .
\end{align*}
By induction one shows that the period of \( x_{l} \DefAs w_{1}^{\infty} \triangleleft \hdots \triangleleft w_{l}^{\infty} \) has length \( p_{l} \DefAs 8 \cdot \prod_{n=2}^{l} ( 2^{n}+2 )\) and contains  \( 2^{l+1} \)-many \( ? \)'s. Since the length of \( w_{l+1} \) is \( 2^{l+1} \cdot 2^{l+1} + 2^{l+2} \), we can fill \( w_{l+1} \) into exactly \( 2^{l+1} + 2 \) copies of the period of \( x_{l} \). We note that the first \( 2^{l} \) copies and the last \( 2^{l} \) copies get completely filled with letters \( a \) and \( b \), while the two middle copies remain completely unfilled. It is now not hard to check that \( x \DefAs \lim_{l \to \infty} x_{l} \) defines a generalised Oxtoby sequence with period structure \( (p_{l})_{l \geq 1} \) . We claim that \( \Card{ \{ j \in \Aper( y ) : y( j ) = b \} } \leq 1 \) holds for every \( y \in X_{x} \). Thus, if \( \Orb(y) \) and \( \Orb(z) \) are proximal orbits, and \( n \in \ZZ \) is such that \( y \) and \( \Shift^{n}( z ) \) are a proximal pair, then \( y \) and \( \Shift^{n}( z ) \) differ in at most two positions: by Proposition~\ref{prop:AperDx} they are equal on \( \ZZ \setminus \Aper( y ) \), and in addition they clearly agree on all positions of \( \Aper( y ) \) where both of them have value \( a \). Hence \( \Orb(y)\) and \( \Orb(z) \) are actually asymptotic orbits. To show the claim, assume that there exist \( y \in X_{x} \) and \( i \neq j \in \Aper( y ) \) with \( y( i ) = y ( j ) = b \). We choose \( l \in \NN \) large enough such that \( \lvert i - j \rvert < p_{l-1} \) holds, and we let \(  (n_{k})_{k} \) denote a sequence with \( \lim_{k\to \infty} \Shift^{n_{k}}( x ) = y \). Because of  \( i , j \in \Aper( y ) \) we also have \( i , j \in \Aper( p_{l} , y ) \), and therefore \( i , j \in \Aper( p_{l} , \Shift^{n_{k}}( x ) ) \) for all sufficiently large \( k \). In particular, \( i+n_{k} \) and \( j+n_{k} \) are \( p_{l} \)-holes in \( x \). Because of \( \lvert i - j \rvert < p_{l-1} \), they lie in the same \( p_{l} \)-block of \( x \) (note that there are by construction no \( p_{l}\)-holes in the beginning and the end of the \( p_{l}\)-blocks). Consequently,  the values of \( \Shift^{n_{k}}( x )( i ) \)  and \(  \Shift^{n_{k}}( x )( j ) \) are given by letters in the same word \( u_{i}^{(n)} \) for a suitable \( n \geq l+1\). Hence at most one of the positions has value \( b \). However, the convergence \( \Shift^{n_{k}}( x ) \to y \) implies \( \Shift^{n_{k}}( x )( i ) =  \Shift^{n_{k}}( x )( j ) = b \) for all large \( k \), yielding the desired contradiction.
\end{exmpl}

\begin{prop}
\label{prop:FiniteFiber}
Let \( X \) be an almost automorphic subshift. We denote by \( \Odo \) its maximal equicontinuous factor and by \( B_{X} \) the boundary of a separating cover that generates \( X \). If there exists \( \omega \in \Odo \) with \( \Card{ \Orb( \omega ) \cap B_{X} } = \infty \) and \( \Card{ \FactOdo_{X}^{-1}( \omega ) } < \infty \), then there are orbits \( \Orb( x ) \), \( \Orb( y ) \) in \( X \) which are proximal but not two-sided asymptotic.
\end{prop}

\begin{proof}
For \( x , y \in \FactOdo_{X}^{-1}( \omega ) \) we define \( A_{x , y } \DefAs \{ j \in \ZZ : x( j ) \neq y( j ) \} \). By \eqref{eq:BndryValues} we have \( \{ j \in \ZZ : \Rot^{j}( \omega ) \in B_{X} \} = \bigcup_{ x ,y \in \FactOdo_{X}^{-1}( \omega ) } A_{ x , y } \). As the union is finite (since \( \FactOdo_{X}^{-1}( \omega ) \) is finite) and the left hand side is infinite, there exist \( x ,y \in \FactOdo_{X}^{-1}( \omega ) \) such that \( A_{ x , y } \) is infinite. In particular, \(  x \) and \( y \) are not a two-sided asymptotic pair. By Proposition~\ref {prop:ElemToOrb} it follows that also the orbits \( \Orb( x ) \), \( \Orb( y ) \) are not asymptotic. It only remains to notice that \( \FactOdo_{X}( x ) = \FactOdo_{X}( y ) \) implies \( \Orb( \FactOdo_{X}( x ) ) = \Orb( \FactOdo_{X}( y ) ) \), and that \( \Orb( x) , \Orb( y ) \) are therefore proximal orbits by Proposition~\ref{prop:ProxiBndry}.
\end{proof}

To give an example of a subshift that satisfies the assumptions of Proposition~\ref{prop:FiniteFiber}, we briefly recall the Oxtoby construction of Williams \cite{Wil_ToepNotUniqErgod}. (A different type of Oxtoby subshift that also satisfies the assumptions of Proposition~\ref{prop:FiniteFiber} will be discussed in Example~\ref{exmpl:OxtWithFactor} and Remark~\ref{rem:ProxiNotAsympt}). For consistency with the rest of our article, we consider only a subset of Williams' examples by imposing two additional restrictions on the construction: firstly, we present the construction and results only for finite alphabets (where \cite{Wil_ToepNotUniqErgod} allows infinite compact alphabets), and secondly we keep the restriction on generalised Oxtoby sequences that at least two intervals per step are not filled (where \cite{Wil_ToepNotUniqErgod} requires only one such interval). Note that this construction also provides examples of one-side asymptotic pairs which are not two sided asymptotic.

\begin{exmpl}[{\cite[Section~3]{Wil_ToepNotUniqErgod}}]
\label{exmpl:WilliamsConstr}
Fix a sequence \( ( a_{l} ) \) in \( \Alphab \) which contains every letter of the alphabet infinitely often, and a sequence \( (p_{l}) \) in \( \NN \) with \( p_{l} \mid p_{l+1} \) and \( \frac{ p_{l} }{ p_{l-1} } \geq 4 \). We start with a completely unfilled, two-sided infinite word and successively fill the holes: in step \( l \), we fill all holes in \( [ -p_{l-1} , -1 ] + p_{l} \ZZ \) and in \( [ 0, p_{l-1} -1 ] + p_{l} \ZZ\) with the letter \( a_{l} \) and leave all other \( p_{l-1} \)-intervals unfilled (because of \( \frac{ p_{l } }{ p_{l-1} } \geq 4 \), there are at least two of them). The result is a sequence \( x \) which is generalised Oxtoby with respect to the period structure \( (p_{l}) \). By \cite[Lemma~3.3]{Wil_ToepNotUniqErgod}, every \( y \in X_{x} \) is constant on \( \Aper( y ) \). In particular, every non-trivial proximal pair \( y \neq z \) differs on all positions of \( \Aper( y ) \). Any such pair with \( \Aper(y) \subseteq \NN \) and \( \Card{ \Aper( y ) } = \infty \) is therefore negatively asymptotic, but not positively asymptotic (and clearly such a pair exists: since generalised Oxtoby subshifts are aperiodic, there is a pair with \( \Aper(y) \subseteq \NN \) and by Proposition~\ref{prop:OxtoNotB}, \( \Aper(y) \) is infinite). Moreover, every \( y \) is uniquely determined by \( \FactOdo_{x}( y ) \in \Odo \) and by the value \( a \in \Alphab \) that \( y \) takes on \( \Aper( y ) \). This implies \( \Card{ \FactOdo_{x}^{-1}( \omega ) } \leq \Card{ \Alphab } < \infty \) for all \( \omega \in \Odo \). As a side note, we remark that in Williams' original setting with a compact alphabet the same reasoning yields an uncountable set of pairwise proximal, non-asymptotic elements. In addition it is worth pointing out that by \cite[Section~5]{Wil_ToepNotUniqErgod}, all subshifts described in this example have entropy zero.
\end{exmpl}

\subsection{The Toeplitz case}

For the remainder of our article, we change our focus from general almost automorphic subshifts to Toeplitz subshifts. We start with several observations relating condition \ref{item:FinOrb} to properties of the subshift.

\begin{prop}
\label{prop:ARegular}
If a Toeplitz subshift satisfies \ref{item:FinOrb}, then it is regular.
\end{prop}

\begin{proof}
The regularity of \( x \) is equivalent to \( B_{X} \) having measure zero for the Haar measure of the odometer, see \cite[Remark~1]{DownIwan_QuasiUnifConv} or \cite[Theorem~13.1]{Downa_OdomToepl}. Since every orbit \( \Orb( \omega ) = \{ \Rot^{n}( \omega ) : n \in \ZZ \} \) is countable, and \( B_{X} \) is contained in finitely many of them, \( B_{X} \) is countable. If \( x \) is non-periodic, we use that every countable subset of the odometer has measure zero, and if \( x \) is periodic, then \( B_{X} \) is empty and hence trivially of measure zero.
\end{proof}

In connection with Proposition~\ref{prop:ARegular}, besides \cite[Remark~1]{DownIwan_QuasiUnifConv} (relating the regularity of \( x \) to measure zero of the cover's boundary) we would also like to mention \cite[Theorem~1]{BJL_ToeplModelSet}. There it is shown that the density of non-periodic positions in \( x \) is equal to the Haar-measure of the CPS window's boundary, thus relating the regularity of \( x \) to measure zero of the CPS window's boundary.

\begin{cor}
\label{cor:UEzeroEntro}
If a Toeplitz subshift satisfies \ref{item:FinOrb}, then it is uniquely ergodic and has topological entropy zero.
\end{cor}

\begin{proof}
This follows immediately from Proposition~\ref{prop:ARegular} and the fact, that a regular Toeplitz subshift is always uniquely ergodic (see \cite[Corollary of Theorem~5]{JacobsKeane_01Toeplitz}) and has entropy zero (that follows directly from the definition of regularity). Of course, entropy zero can also be deduced from Corollary~\ref{cor:LiYorkeEntroZero}.
\end{proof}

Note that Proposition~\ref{prop:ARegular} and Corollary~\ref{cor:UEzeroEntro} remain true when~\ref{item:FinOrb} is replaced with the weaker requirement that at most countably many orbits \( \Orb( \omega ) \subseteq \Omega \) intersect \( B_{X} \). We also remark that regularity of a Toeplitz word does not imply~\ref{item:FinCap}, as shown for instance by the existence of regular Oxtoby words (see Example~\ref{exmpl:OxtProxiAsym}).

\begin{prop}
\label{prop:BdHolesFiniteD}
Let \( x \) be a Toeplitz word and let \( B_{X} \) denote the boundary of the separating cover that generates \( X_{x} \). If there exists \( h \in \NN \) with \( \Card{ \Aper( p_{l} , x ) \cap [ 0, p_{l}-1 ] } \leq h \) for all \( l \in \NN \), then \( \Card{ B_{X} } \leq h \) follows, and in particular condition \ref{item:FinD} holds.
\end{prop}

\begin{proof}
Let \( \omega \in B_{X} \) and \( y \in \FactOdo_{x}^{-1}( \omega ) \) be arbitrary. By Proposition~\ref{prop:AperDx}~\ref{item:A1} and by the definition of \( \FactOdo_{x} \) this implies for all \( l \in \NN \)
\[ 0 \in \Aper( y ) \subseteq \Aper( p_{l} , y ) = \Aper( p_{l} , x ) - \omega( l ) . \]
Consequently, we have \( \omega( l ) \in \Aper( p_{l} , x ) \cap [0 , p_{l}-1] \). Since \( \omega \in B_{X} \) was arbitrary, this yields
\[ B_{X} \subseteq \bigcup_{j \in \Aper( p_{l} , x ) \cap [ 0, p_{l}-1 ] } [ \Embed( j ) ]_{l} . \]
Since \( \Card{ \Aper( p_{l} , x ) \cap [ 0, p_{l}-1 ] } \leq h \) holds by assumption, the right hand side is a union over at most \( h \) cylinder sets in \(\Odo \), and as each of them converges to a singleton, $ \Card{ B_X } \leq h$ follows.
\end{proof}

However, the converse of Proposition~\ref{prop:BdHolesFiniteD} is not true, as the following example shows.

\begin{exmpl}
\label{ex:UnbdHolesOneBdry}
We construct a Toeplitz sequence \( x \in \{ a, b \}^{\ZZ} \) for which the boundary of the separating cover is a singleton (in particular condition~\ref{item:FinD} holds), while \( \Card{ \Aper( p_{l} , x ) \cap [ 0, p_{l}-1 ] } \) is unbounded. We define \( x \DefAs \lim_{n \to \infty} w_{1}^{\infty} \triangleleft \hdots \triangleleft w_{l}^{\infty} \) via the hole-filling procedure from the following finite words with holes
\begin{align*}
w_{1} & \DefAs a^{1} (? ?)^{1} b^{1} = a??b &
w_{2} & \DefAs (aa)^{1} (?a)^{1} (?b)^{1} (bb)^{1} = aa?a?bbb\\
w_{3} & \DefAs a^{2} (??)^{2} b^{2} = aa????bb &
w_{4} & \DefAs (aa)^{2} (?a)^{2} (?b)^{2} (bb)^{2}\\
&&& \hphantom{:}= aaaa?a?a?b?bbbbb\\
w_{2l-1} & \DefAs a^{2^{l-1}} (??)^{2^{l-1}} b^{2^{l-1}} &
w_{2l} & \DefAs (aa)^{2^{l-1}} (?a)^{2^{l-1}} (?b)^{2^{l-1}} (bb)^{2^{l-1}}
\end{align*}
It is easily checked that \( w_{1}^{\infty} \triangleleft \hdots \triangleleft w_{l}^{\infty} \) is \( 4^{l} \)-periodic, and that \( (4^{l})_{l \in \NN} \) is actually a period structure. Figure~\ref{fig:UnbdHolesOneBdry} shows the associated odometer, with labels on the cylinder sets indicating which set of the separating cover they belong to (equivalently: what is their value under the semicocycle). Note how the right half of all sets that were undetermined on level \( 2l-2 \), is determined on level \(2l\), resulting thus in a single boundary point. Moreover, the number of undetermined cylinder sets is the same on level \( 2l-1 \) and level \( 2l \), but doubles from level \( 2l \) to level \( 2l+1 \).

\begin{figure}
\centering
\scriptsize
\pgfmathsetmacro{\TikzHeight}{-0.5} 
\pgfmathsetmacro{\TikzDist}{3} 
\pgfmathsetmacro{\TikzLong}{(0.95*\linewidth - 3*\TikzDist) / 4} 
\pgfmathsetmacro{\TikzMedium}{(\TikzLong - 3*\TikzDist) / 4} 
\pgfmathsetmacro{\TikzShort}{(\TikzMedium - 3*\TikzDist) / 4} 
\begin{tikzpicture}
[
CylinFill/.pic = {\draw (0,0) --node[midway, above]{a} ++( #1 pt  , 0 ) ++( \TikzDist pt, 0) --node[midway, above]{a} ++( #1 pt  , 0 ) ++( \TikzDist pt, 0) --node[midway, above]{b} ++( #1 pt  , 0 ) ++( \TikzDist pt, 0) --node[midway, above]{b} ++( #1 pt  , 0 );},
CylinUnfill/.pic = {\draw ( 0 , 0 ) --node[midway, above]{a} ++( #1 pt  , 0 ) ++( \TikzDist pt, 0) -- ++( #1 pt  , 0 ) ++( \TikzDist pt, 0) -- ++( #1 pt  , 0 ) ++( \TikzDist pt, 0) --node[midway, above]{b} ++( #1 pt  , 0 );}]
\path (0,0) pic{CylinUnfill=\TikzLong};
\path (\TikzLong + \TikzDist pt, \TikzHeight) pic{CylinUnfill=\TikzMedium} ++(\TikzLong + \TikzDist pt, 0) pic{CylinFill=\TikzMedium};
\path (\TikzLong + \TikzMedium + 2*\TikzDist pt, 2*\TikzHeight) pic{CylinUnfill=\TikzShort} +(2*\TikzShort + \TikzDist pt, 0)node[below](A){}  +(3*\TikzShort + 2*\TikzDist pt, 0)node[below](B){} ++ ( \TikzMedium + \TikzDist pt, 0) pic{CylinUnfill=\TikzShort} +(\TikzShort + \TikzDist pt, 0)node[below](C){} +(2*\TikzShort + 2* \TikzDist pt, 0)node[below](D){} ;
\path (2*\TikzMedium + 2*\TikzDist pt , 3.7*\TikzHeight) pic{CylinUnfill=\TikzShort} +(0.5*\TikzMedium pt, 0)node[above=-0.5*\TikzHeight](X){} ++( 2*\TikzMedium + 2*\TikzDist pt , 0 ) pic{CylinUnfill=\TikzShort} +(0.5*\TikzMedium pt, 0)node[above=-0.5*\TikzHeight](Y){} ++( 3*\TikzMedium + 3*\TikzDist pt , 0 ) pic{CylinFill=\TikzShort} +(0.5*\TikzMedium pt, 0)node[above=-0.5*\TikzHeight](Z){} ++( 2*\TikzMedium + 2*\TikzDist pt , 0 ) pic{CylinFill=\TikzShort} +(0.5*\TikzMedium pt, 0)node[above=-0.5*\TikzHeight](W){};
\draw (A)--(X) (B)--(Y) (C)--(Z) (D)--(W);
\end{tikzpicture}
\normalsize
\caption{\label{fig:UnbdHolesOneBdry}The separating cover \( \Odo = C_{a} \cup C_{b} \) that generates Example~\ref{ex:UnbdHolesOneBdry}. The bottom line shows (zoomed in) the level-4 cylinder sets inside \( [1 , 5 , 21] \), \( [1 , 5 , 37] \), \( [1 , 9 , 25] \) and \( [1 , 9 , 41] \) respectively.}
\end{figure}
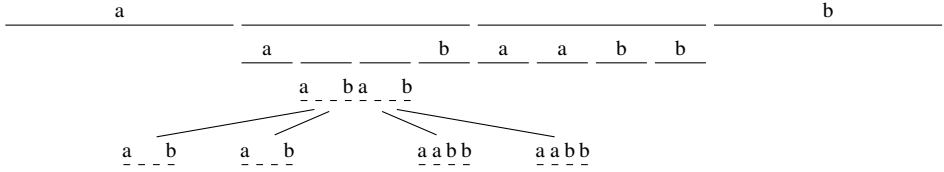

Making the above precise, we define \( U_{l} \DefAs \{ u \in \{ 1, 2 \}^{2l} : u(1) = \hdots = u(l) = 1 \} \) and claim
\[ \Aper( 4^{2l} , x ) = \bigcup_{u \in U_{l}} ( \sum_{i=1}^{2l} u(i) \cdot 4^{i-1} + 4^{2l}\ZZ) \subseteq 4^{0} + 4^{1} + \hdots + 4^{l-1} + 4^{l} \ZZ , \]
which we will prove in a moment. On the one hand, \( \Card{ U_{l} } = 2^{l} \) then implies that \( \Card{ \Aper( 4^{2l}, x ) \cap [0, 4^{2l} - 1] } = 2^{l} \) holds for each \( l \in \NN \), so the number of holes per period is unbounded. On the other hand, the same argument as in the proof of Proposition~\ref{prop:BdHolesFiniteD} shows that \( B_{X} = \{ ( 1 , 5 , 21 , \hdots  , \sum_{i=0}^{l} 4^{i} , \hdots ) \} \) is a singleton. We now prove our claim by induction: for \( l = 1 \) we have by definition \( w_{1}^{\infty} \triangleleft w_{2}^{\infty} = (aaaba?aba?bbabbb)^{\infty} \) and hence \( \Aper( 4^{2} , x ) = ( 1 \cdot 4^{0} + 1 \cdot 4^{1}  + 4^{2} \ZZ ) \cup ( 1 \cdot 4^{0} + 2 \cdot 4^{1}  + 4^{2} \ZZ ) \). Assume now that the claim holds for some \( l \in \NN \). We proceed in two steps: first we insert \( w_{2l+1}^{\infty} = ( a^{2^{l}} (??)^{2^{l}} b^{2^{l}} )^{\infty} \) into the positions
\begin{align*}
&\hspace{-3.7em} \Aper( 4^{2l} , x ) = \bigcup_{u \in U_{l}} \Big( \sum_{i=1}^{2l} u(i) \! \cdot \! 4^{i-1} + 4^{2l}\ZZ \Big)\\
= \bigcup_{u \in U_{l}} \Big( &( \sum_{i=1}^{2l} u(i) \! \cdot \! 4^{i-1} + 0 \! \cdot \! 4^{2l} + 4^{2l+1} \ZZ ) \cup ( \sum_{i=1}^{2l} u(i) \! \cdot \! 4^{i-1} + 1 \! \cdot \! 4^{2l} + 4^{2l+1} \ZZ ) \\
{} \cup {} & ( \sum_{i=1}^{2l} u(i) \! \cdot \! 4^{i-1} + 2 \! \cdot \! 4^{2l} + 4^{2l+1} \ZZ ) \cup ( \sum_{i=1}^{2l} u(i) \! \cdot \! 4^{i-1} + 3 \! \cdot \! 4^{2l} + 4^{2l+1} \ZZ ) \Big) .
\end{align*}
Because of \( \Card{ U_{l} } = 2^{l} \), the positions from the first expression are precisely those that get filled with \( a \), the positions of the two middle expressions remain completely unfilled, and the positions from the last expression get filled with \( b \). We obtain
\begin{align*}
\Aper( 4^{2l+1} , x ) = \bigcup_{u \in U_{l}} \Big( & ( \sum_{i=1}^{2l} u(i) \cdot 4^{i-1} + 1 \cdot 4^{2l} + 4^{2l+1} \ZZ ) \\
{} \cup {} & ( \sum_{i=1}^{2l} u(i) \cdot 4^{i-1} + 2 \cdot 4^{2l} + 4^{2l+1} \ZZ ) \Big) .
\end{align*}
In the second step, we now split these positions into residue classes modulo \( 4^{2l+2} \) and insert \( w_{2l+2}^{\infty} = ( (aa)^{2^{l}} (?a)^{2^{l}} (?b)^{2^{l}} (bb)^{2^{l}} )^{\infty} \). Again, the positions from the first and from the last expression get completely filled, which ensures that \( \Aper( 4^{2l+2}, x ) \) can be encoded by words \( u \in \{ 1, 2 \}^{2l+2} \). Moreover, the alternation of \( ? \)'s and non-\( ? \)-letters in the remaining positions implies \( u(l+1) = 1 \), so that \( \Aper( 4^{2l+2}, x ) \) is indeed described by \( U_{l+1} \).
\end{exmpl}

Next we provide an example which shows that property~\ref{item:FinD} does not imply a linear bound on the word complexity. More precisely, our example has a single hole per period (hence property~\ref{item:FinD}, see Proposition~\ref{prop:BdHolesFiniteD}), but nevertheless superpolynomial complexity along a subsequence (that is, \( \limsup_{L\to\infty} \frac{\cplx{x}(L)}{\left\lvert q(L)\right\rvert}  = \infty\) for every polynomial $q$). In Remark~\ref{rem:NonSuperlinCompl} we discuss that such subshifts also have non-superlinear complexity (that is, \( \liminf_{L \to \infty} \frac{ \cplx{x}( L ) }{ L } < \infty \)), which shows that different complexity behaviours can coexist in the same word. A different example of this phenomenon can be found in \cite[Section~4.1]{DonoDuraMaaPeti-AutoGrLowCompl}.

\begin{exmpl}
\label{exmpl:deBrujin}
We construct a Toeplitz sequence in \( \{ a, b \}^{\ZZ} \) via hole filling: for \( l \geq 3 \) let \( w_{B ,\, l} \) denote a de Bruijn word of order \( l! \), that is, a word of length \( 2^{l!} \) that contains every word of length \( l! \) when read cyclic. We choose the starting point of \( w_{B , \, l} \) such that it ends with \( b^{l!} \). Then we replace one letter of \( b^{l!} \) (neither the first nor the last) by \( ? \), and call the resulting word \( w_{l} \), for example
\[ w_{3} = aaaaaabaaaabbaaababaaabbbaabaababbaabbabaabbbbabababbbabbabb?bbb . \]
We write \( x \DefAs \lim_{ l \to \infty } w_{3}^{\infty} \triangleleft \hdots \triangleleft w_{l}^{\infty} \) for the Toeplitz sequence generated by this process. Moreover, we set \( p_{l} \DefAs \prod_{n=3}^{l} \lvert w_{n} \rvert = \prod_{n=3}^{l} 2^{n!} \) and note that \( \Card{ \Aper( p_{l} , x ) \cap [0 , p_{l}-1] } = 1 \) holds. We denote the (\( p_{l} \)-periodic) word between two consecutive \( \Aper(p_{l} , x ) \)-positions by \( W_{l} \). Thus \( x \) can be decomposed as \( \hdots W_{l} \star W_{l} \star W_{l} \star \hdots \), with each \( \star \) denoting a letter from \( \Alphab \). Every word of length \( (l+1)! \cdot p_{l} \) is therefore of the form
\[ v_{l}( j , u ) \DefAs ( W_{l} u_{1} W_{l} \hdots u_{(l+1)!} W_{l} )[ j , j + (l+1)! \cdot p_{l} - 1 ] , \]
with \( u \in \{ a , b \}^{(l+1)!} \) and \( j \in [1, p_{l} ] \). Conversely, all \( v_{l}( j , u ) \) appear in \( x \) by the de Bruijn property of \( w_{l+1} \). To show that \( (p_{l} ) \) is a period structure of \( x \), we consider \( v_{l}( 1 ,  a \hdots a ) \), which is contained in \( W_{l+1} \) and appears therefore \( p_{l+1} \)-periodically. By counting the number of \( a \)'s in any word of length \( (l+1)! \cdot p_{l} \), it follows from the decomposition \( x = \hdots W_{l} \star W_{l} \hdots \) that \( v_{l}( 1 ,  a \hdots a ) \) appears only where \( (l+1)! \)-many consecutive \( p_{l} \)-holes have value \( a \). By the de Bruijn property, \( a^{(l+1)!} \) appears only once in \( w_{l+1} \), so \( p_{l+1} \) is indeed the shortest period for \( v_{l}( 1 ,  a \hdots a ) \), and hence also for \( W_{l+1} \). Moreover, \( x \) has by construction only a single \( p_{l} \)-hole per period. To show superpolynomial complexity along a subsequence, we claim that
\[ \cplx{x}( (l+1)! \cdot p_{l} ) \geq 2^{(l+1)!} p_{l} \]
holds. We prove the claim by showing that the \( 2^{(l+1)!} p_{l} \) possible words \( v_{l}( j , u ) \) are pairwise different. Assume hence \( v_{1} \DefAs v_{l}( j_{1} , u_{1} ) = v_{l}( j_{2} , u_{2} ) \AsDef v_{2} \), and let \( m \leq l \) be maximal such that \( j_{1} \equiv j_{2} \mod p_{m} \) holds. Thus, the \( p_{m} \)-holes are in the same positions in \( v_{1} \) and \( v_{2} \), and \( \lvert v_{i} \rvert \geq (m+1)! \cdot p_{m} \) implies that both \( v_{i} \)'s contain at least \( (m+1)! \)-many of them. Because of \( v_{1} = v_{2} \), the \( p_{m} \)-holes in \( v_{1} \) and \( v_{2} \) are filled with the same word of length \( (m+1)! \). By the de Bruijn property of \( w_{m+1} \), every word of length \( (m+1)! \) appears only once along the \( p_{m} \)-holes within \( W_{m+1} \), so \( v_{1} \) and \( v_{2} \) must occur at the same place within \( W_{m+1} \). The maximality of \( m \) hence implies \( m = l \), and thus \( j_{1} = j_{2} \) and \( u_{1} = u_{2} \) as claimed. Having thus proved the lower bound on the complexity, we now check that \( x \) has superpolynomial complexity along the subsequence \( ( (l+1)! \cdot p_{l} )_{ l \in \NN } \): firstly, we note that the estimate
\[ p_{l} = \prod_{n=3}^{l} 2^{n!} = 2^{ l! + (l-1)! + \hdots +3! } \leq 2^{l! + (l-1)! \cdot (l-3)} < 2^{l! \cdot 2} \]
holds. Let now \( q \) be any polynomial, and let \( m \in \NN \) be such that \( \lvert q(n) \rvert \leq n^{m} \) holds for all sufficiently large \( n \). For sufficiently large \( l \) we then obtain
\[ \frac{ \cplx{x}( (l+1)! \cdot p_{l} ) }{ \lvert q( (l+1)! \cdot p_{l} ) \rvert } \geq \frac{ 2^{(l+1)!} p_{l} }{ ((l+1)!)^{m} p_{l}^{m} } > \frac{ 2^{l! \cdot (l+1)} \cdot 2^{l!} \cdot 2^{(l-1)! + \hdots 3! } }{ (l!)^{m} \cdot (l+1)^{m} \cdot 2^{ l! \cdot 2m }} \xrightarrow{l \to \infty} \infty . \]
\end{exmpl}

\begin{rem}
\label{rem:NonSuperlinCompl}
While Example~\ref{exmpl:deBrujin} shows superpolynomial complexity along one subsequence of positions, every Toeplitz sequence with a single hole per period also has non-superlinear complexity. To see that this is the case, let \( ( p_{l} ) \) denote a period structure and let \( W_{l} \) be the word of length \( p_{l} - 1 \) between two consecutive positions of \( \Aper( p_{l} , x ) \). Then \( x \) can be written as \( x = \hdots W_{l} \star W_{l} \star W_{l} \star \hdots \), with each \( \star \) denoting a letter from \( \Alphab \). Hence every word of length \( p_{l} \) in \( x \) is contained in some \( W_{l} \star W_{l} \). Since there are \( \lvert W_{l} \rvert + 1 = p_{l} \) possibilities for the starting point, and \( \Card{\Alphab} \)  possibilities for the value of \( \star \), we obtain \( \cplx{x}( p_{l} ) \leq p_{l} \cdot \Card{ \Alphab } \) and thus
\[ \liminf_{L\to\infty} \frac{\cplx{x}(L)}{L} \leq \lim_{l\to\infty} \frac{\cplx{x}(p_{l})}{p_{l}} \leq \frac{ p_{l} \cdot \Card{ \Alphab } }{ p_{l} } = \Card{ \Alphab }  < \infty .  \]
Similarly, assume that the number of holes per period is bounded, that is, there exists \( h \in \NN \) with \( \Card{ \Aper( p_{l} , x ) \cap [ 0, p_{l}-1 ] } \leq h \) for all \( l \in \NN \). Then \( \liminf_{L \to \infty} \frac{\cplx{x}(L)}{L} \leq \Card{ \Alphab }^{h} \) follows. We do not know to which extend non-superlinear complexity holds in general for Toeplitz subshifts with property~\ref{item:FinD} (which is a strictly weaker condition than a bounded number of holes per period, see Proposition~\ref{prop:BdHolesFiniteD} and Example~\ref{ex:UnbdHolesOneBdry}).
\end{rem}

\section{When a Toeplitz subshift has a factor with property~\ref{item:FinD}}
\label{sec:FactFiniteProxi}

As we have seen in Corollary~\ref{cor:PresUnderFact}, property~\ref{item:FinD} is preserved when going from a subshift to a factor subshift with the same maximal equicontinuous factor. However, when \ref{item:FinD} fails for a subshift, there may or may not be a factor subshift with~\ref{item:FinD} (note that this factor is then necessarily proper, since any conjugacy would preserve \ref{item:FinD}). In this section we discuss criteria for the existence or non-existence of such factors in the Toeplitz case. Recall from Section~\ref{subsec:PrelimToepl} that a factor subshift \( \Psi( X_{x} ) \) of a Toeplitz subshift \( X_{x} \) is the Toeplitz subshift \( X_{\Psi( x )} \), and that the factor map \( \Psi \) is given by a sliding block code. Recall also that for every \( y \in X_{x} \), the shift of \( x \) relative to \( y \) at the periodic positions is uniquely determined, see~\eqref{eqn:UniqueShift}, and that this defines the factor map \( \FactOdo_{x} \colon X_{x} \to \Omega \) to the maximal equicontinuous factor. 

In this section, we will use the language of semicocycles (see Section~\ref{subsec:SepCover}) rather than that of separating covers. This allows us to know not only whether \( \omega \in \Odo \) is a boundary point of the separating cover, but through \( F_{x}(\omega) \) we also keep track of which of the cover's sets contain \( \omega \). This does not matter for binary alphabets (there are only two covering sets, and boundary points belong to both of them), but it will be important for larger alphabets and for controlling boundary points of a factor subshift. For instance, if the letters \( b \) and \( c \) are identified in a factor subshift, then a boundary point which is only contained in \( C_{b} \) and \( C_{c} \) will vanish, while a boundary point that also belongs to \( C_{a} \) might be preserved. We denote the set of discontinuities of \( \tau_{x} \) respectively \( \tau_{\Psi( x )} \) by \( B_{X} \) respectively \( B_{\Psi( X )} \).

\subsection{Sufficient condition}

Our first aim is to construct, under certain conditions, a factor subshift with a single semicocycle discontinuity. The main idea is to find a discontinuity point \( \omega \) and finite words \( u_{n} \), such that \( \omega \) is the only discontinuity point on the boundary of the projection \( \FactOdo_{x}( [ u_{1}] \cup [ u_{2} ] \cup \hdots \} ) \) of the cylinder sets of the \( u_{n} \)'s. The sliding block code that maps all \( u_{n} \)'s to \( a \) and everything else to \( b \), has then constant value around any discontinuity point except \( \omega \). Therefore \( \omega \) is the only discontinuity point that is preserved under the factor map, and in fact the unique semicocycle discontinuity of the factor subshift. Before we formulate and prove this result rigorously, we show an auxiliary lemma that will allows us to identify a suitable finite word \( u \) and which might be of independent interest. We remark that in \cite[Lemma~5]{CassKar_ToeplWords} a related result can be found for the special case of so-called \((p,q)\)-Toeplitz words (see \cite{CassKar_ToeplWords} for the definition).

\begin{lem}
\label{lem:UniquePosModP}
Let \( x \in \Alphab^{\ZZ} \) be a non-periodic Toeplitz word with period structure \( (p_{l}) \). For every \( l_{1} \in \NN \) there exists \( l_{2} \in \NN \) such that each word with length at least \( p_{l_{2}} \) occurs in \( x \) only in a unique residue class modulo \( p_{l_{1}} \). In other words: \( x[ j_{1} , j_{1} + p_{l_{2}} - 1 ] = x[  j_{2} , j_{2} + p_{l_{2}} - 1 ] \) implies \( j_{1} \equiv j_{2} \mod p_{l_{1}} \).
\end{lem}

\begin{proof}
We set \( \{ n_{1} , \hdots , n_{K} \} \DefAs \Aper( p_{l_{1}} , x ) \cap [ 0 , p_{l_{1}}-1 ] \). For each \( n_{k} \), we fix \( m_{k} \in n_{k} + p_{l_{1}} \ZZ \) with \( x( m_{k} ) \neq x( n_{k} ) \). We choose \( l_{2} \in \NN \) large enough such that \( \{ n_{1} , \hdots , n_{K} \} \cup \{ m_{1} , \hdots , m_{K} \} \subseteq \Per( p_{l_{2}} , x ) \) holds. Let now \( j_{1}, j_{2} \in \ZZ \) be such that \( x[ j_{1} , j_{1} + p_{l_{2}} - 1 ] = x[ j_{2} , j_{2} + p_{l_{2}} - 1 ] \) holds. We claim that this implies \( \Aper(p_{l_{1}} , x ) + j_{2} - j_{1} \subseteq \Aper( p_{l_{1}} , x ) \): for fixed \( k \in \{ 1 , \hdots , K \} \), denote by \( \widetilde{n} \) respectively \( \widetilde{m} \) the unique element from \( [ j_{1} , j_{1} + p_{l_{2}} - 1 ] \) that lies in \( n_{k} + p_{l_{2}} \ZZ \) respectively \( m_{k} + p_{l_{2}} \ZZ \). 
We obtain
\begin{align*}
x( \widetilde{n} + j_{2} - j_{1} ) &= x( \widetilde{n} ) &&\hspace{-12em} \text{since }   x[ j_{2} , j_{2} + p_{l_{2}} - 1 ] = x[ j_{1} , j_{1} + p_{l_{2}} - 1 ], \\
{}&= x( n_{k} ) &&\hspace{-12em} \text{since } \widetilde{n} \in n_{k} + p_{l_{2}} \ZZ \subseteq \Per( p_{l_{2}} , x ), \\
&\neq x( m_{k} ) = x( \widetilde{ m } ) = x ( \widetilde{m}  + j_{2} - j_{1}) .
\end{align*}
Since \(  \widetilde{n} + j_{2} - j_{1} \) and \(  \widetilde{m} + j_{2} - j_{1} \) both are in \(  n_{k} + j_{2} - j_{1} + p_{l_{1}} \ZZ \), it follows as claimed for every \( n_{k} + p_{l_{1}} \ZZ \subseteq \Aper( p_{l_{1}} , x ) \) that \( x \) is not constant on \(  n_{k} + j_{2} - j_{1} + p_{l_{1}} \ZZ \).

Next we note that \( \Aper(p_{l_{1}} , x ) + j_{2} - j_{1} \subseteq \Aper( p_{l_{1}} , x ) \) implies \( \Aper( p_{l_{1}} , x ) + j_{2} - j_{1} = \Aper( p_{l_{1}} , x ) \),  since \( \Aper(p_{l_{1}} , x ) \) is a \( p_{l_{1}} \)-periodic set. After taking complements, we obtain
\begin{equation}
\label{eqn:PerSetsShifted}
\Per( p_{l_{1}} , \Shift^{j_{1}}( x ) ) = \Per( p_{l_{1}} , x ) - j_{1} = \Per( p_{l_{1}} , x ) - j_{2} = \Per( p_{l_{1}} , \Shift^{j_{2}}( x ) ) .
\end{equation}
Moreover, \( x[ j_{1} , j_{1} + p_{l_{2}} - 1 ] = x[ j_{2} , j_{2} + p_{l_{2}} - 1 ] \) implies \( \Shift^{j_{1}}( x )[ 0 , p_{l_{2}} - 1 ] = \Shift^{j_{2}}( x )[ 0 , p_{l_{2}} - 1 ] \). As \( \Shift^{j_{1}}( x ) \) and \( \Shift^{j_{2}}( x ) \) agree on an interval that is longer than \( p_{l_{1}} \), and since the \( p_{l_{1}} \)-periodic positions are equal by Equation~\eqref{eqn:PerSetsShifted}, we get
\[ \Shift^{j_{1}}( x )( \Per( p_{l_{1}} , \Shift^{j_{1}}( x ) ) ) = \Shift^{j_{2}}( x )( \Per( p_{l_{1}} , \Shift^{j_{2}}( x ) ) ) . \]
It only remains to notice that the \( p_{l_{1}} \)-periodic part of \( x \) is uniquely determined modulo \( p_{l_{1}} \), see Equation~\eqref{eqn:UniqueShift}.
\end{proof}

\begin{thm}
\label{thm:IsoDiscontFact}
Let \( x \in \Alphab^{\ZZ} \) be a non-periodic Toeplitz word and let \( \Omega \) denote the associated odometer. Assume \( \omega \in B_{X} \) and \( a \in \Alphab \) are such that \( ( \omega , a ) \in F_{x} \) is an isolated point in \( \bigcup_{ \widetilde{\omega} \in B_{X} } F_{x}( \widetilde{\omega} ) \). Then there exists a factor subshift \( X_{\Psi(x)} \) of \( X_{x} \) with the same associated odometer \( \Omega \) and with \( B_{\Psi( X )} = \{ \omega \} \).
\end{thm}

\begin{proof}
Let \( ( p_{l} ) \) denote a period structure of \( x \). Since by assumption \( ( \omega , a ) \) is isolated in \( \bigcup_{ \widetilde{\omega} \in B_{X} } F_{x}( \widetilde{\omega} ) \), there exists \( l_{1} \in \NN \) with \( ( [ \omega ]_{ l_{1} } \times \{ a \} ) \cap ( \bigcup_{ \widetilde{\omega} \in B_{X} } F_{x}( \widetilde{\omega} )  ) = \{ ( \omega , a ) \} \). We choose \( l_{2} \) according to Lemma~\ref{lem:UniquePosModP} large enough, such that words with length at least \( p_{l_{2}} \) have a uniquely determined position modulo \( p_{l_{1}} \) in \( x \). We consider the set
\[ \mathscr{U} \DefAs \{ x[ j-p_{l_{2}} , j + p_{l_{2}} ] : j \in \omega( l_{1} ) + p_{l_{1}} \ZZ  \text{ with } x( j ) = a \} , \]
which contains sufficiently long words in \( x \), with central letter \( a \) and appearing around the positions \( \omega( l_{1} ) + p_{l_{1}} \ZZ \). By the aforementioned uniqueness property, the words from \( \mathscr{U} \) appear only around these positions. We define the following sliding block code:
\[ \psi \colon \, \Alphab^{[-p_{l_{2}} , p_{l_{2}}]} \to \Alphab \, , \; u \mapsto \begin{cases}
a & \text{if } u \in \mathscr{U},\\
b & \text{otherwise.}
\end{cases} \]
Let \( \Psi \colon X_{x} \to \{ a , b \}^{\ZZ} \) denote the factor map defined by \( \psi \). We have \( \Psi( x )( j ) = a \) if and only if \( j \in \omega( l_{1} ) + p_{l_{1}} \ZZ \) and \( x( j ) = a \) hold. By Proposition~\ref{prop:OdoFactIso}, the associated odometer \( \Omega^{\prime} \) of the factor subshift \( X_{\Psi( x )} \) is a factor of \( \Omega \). In particular, if \( (q_{m}) \) denotes a period structure of \( \Omega^{\prime} \), then for every \( q_{m} \) there exists \( p_{l} \) with \( q_{m} \mid p_{l} \) (see Proposition~\ref{prop:OdoFactIso} again).

Next we show that conversely, for every \( p_{l} \) there exists \( q_{m} \) with \( p_{l} \mid q_{m} \), thus proving \( \Omega^{\prime} = \Omega \). First we use that \( \omega \) is a discontinuity point with \(  ( \omega , a ) \in F_{x} \). For every \( l \in \NN \) it follows therefore from Proposition~\ref{prop:AperDx} and the definition of \( \FactOdo_{x} \) that \( \omega( l ) + p_{l} \ZZ \) contains positions where the value of \( x \) is \( a \) and positions where the value of \( x \) is not \( a \). For \( l \geq l_{1} \) this implies that \( \Psi( x ) \) takes values \( a \) and \( b \) on \( \omega( l ) + p_{l} \ZZ \), which yields
\begin{equation}
\label{eqn:OmegInAper}
\omega( l ) + p_{l} \ZZ \subseteq \Aper( p_{l} , \Psi( x )  ) \quad \text{for all } l \geq l_{1} . 
\end{equation}
Secondly, let \( l > l_{1} \) be arbitrary and recall that \( l_{1} \) was chosen such that \( \omega \) is the only discontinuity point of \( \tau_{x} \) with value \( a \) in \( [ \omega ]_{l_{1}} \). Thus, for every \( \widetilde{ \omega } \in [ \omega ]_{l_{1}} \setminus [ \omega ]_{l} \) there exists \( l( \widetilde{\omega} ) \in \NN \) such that the value of \( \tau_{x} \) on \( \Embed( \ZZ ) \cap [ \widetilde{\omega} ]_{l( \widetilde{\omega} )} \) is either never equal to \( a \), or is constant \( a \). Without loss of generality, we may assume \( l( \widetilde{\omega} ) \geq l \). The sets \( [ \widetilde{\omega} ]_{l( \widetilde{\omega} )} \), with \( \widetilde{ \omega } \in [ \omega ]_{l_{1}} \setminus [ \omega ]_{l} \), form an open cover of the compact set \(  [ \omega ]_{l_{1}} \setminus [ \omega ]_{l} \). Hence there is a finite subcover. In particular, there exists \( \widehat{l} \geq l \) (given by the largest \( l( \widetilde{\omega} ) \) in the subcover) and finitely many \( \widehat{l} \)-cylinder sets partitioning \(  [ \omega ]_{l_{1}} \setminus [ \omega ]_{l} \), such that \( \tau_{x} \) is either never equal to \( a \), or is constant \( a \) on each of them. In other words, every arithmetic progression \( s + p_{l} \ZZ \subseteq \omega( l_{1} ) + p_{l_{1}} \ZZ \) with \( s \neq \omega( l ) \) consists of \( p_{\, \widehat{l}\,} \)-progressions on which \( x \) is either never equal to \( a \), or is constant \( a \). Consequently, \( \Psi(x) \) is \( p_{\, \widehat{l}\,} \)-periodic on each \( s + p_{l} \ZZ \subseteq \omega( l_{1} ) + p_{l_{1}} \ZZ \) with \( s \neq \omega( l ) \). Since additionally \( \Psi( x ) \) has constant value \( b \) outside of \( \omega( l_{1} ) + p_{l_{1}} \ZZ \) (and is thus \( p_{l_{1}} \)-periodic there), we obtain
\begin{equation}
\label{eqn:AperInOmeg}
\Aper( p_{\, \widehat{l}\,} , \Psi( x )  ) \subseteq \omega( l ) + p_{l} \ZZ .
\end{equation}
Moreover, by~\eqref{eqn:OmegInAper} there are \( j_{1} , j_{2} \in \omega( \widehat{l} ) + p_{\, \widehat{l}\,} \ZZ \) with \( \Psi( x )( j_{1} ) \neq \Psi( x )( j_{2} ) \). Since \( (q_{m} ) \) is a period structure of \( \Psi( x ) \), there exists \( q_{m} \) such that \(  j_{1} , j_{2} \in \Per( q_{m} , \Psi( x ) ) \) holds. To finish the argument, we combine \( j_{1} + q_{m} \equiv j_{2} + q_{m} \mod p_{\, \widehat{l}\,} \) and
\[ \Psi( x )( j_{1} + q_{m} ) = \Psi( x )( j_{1} ) \neq \Psi( x )( j_{2} ) = \Psi( x )( j_{2} + q_{m} ) \]
to conclude \( j_{1}+q_{m} \in \Aper( p_{\, \widehat{l}\,} , \Psi( x ) ) \subseteq  \omega( l ) + p_{l} \ZZ \), see~\eqref{eqn:AperInOmeg}. Since \( j_{1} \in \omega( \widehat{l} ) + p_{\, \widehat{l}\,} \ZZ \subseteq \omega( l ) + p_{l} \ZZ \) holds by definition of \( j_{1} \), we obtain \( p_{l} \mid q_{m} \) as claimed, proving that \( X_{x} \) and \( X_{\Psi(x)} \) indeed have the same associated odometer \( \Omega \).

It only remains to show that \( B_{\Psi( X )} = \{ \omega \} \) holds: on the one hand, Equation~\eqref{eqn:OmegInAper} implies \( \{ \omega \} \subseteq B_{\Psi( X )} \). On the other hand, Equation~\eqref{eqn:AperInOmeg} implies \( B_{\Psi( X )} \subseteq [\omega]_{l} \) for all \( l > l_{1} \), and hence \( B_{\Psi( X )} \subseteq \{ \omega \} \).
\end{proof}

\begin{rem}
Instead of a single discontinuity point \( \omega \in B_{X} \) for which \( ( \omega , a ) \) is isolated in \( \bigcup_{ \widetilde{\omega} \in B_{X} } F_{x}( \widetilde{\omega} ) \), we could consider a version of Theorem~\ref{thm:IsoDiscontFact} with a finite set \( \{ \omega_{1} , \hdots , \omega_{N} \} \subseteq B_{X} \) and values \( a_{1} , \hdots , a_{N} \in \Alphab \) such that each \( ( \omega_{n} , a_{n} ) \) is isolated in \( \bigcup_{ \widetilde{\omega} \in B_{X} } F_{x}( \widetilde{\omega} ) \). As in the proof of Theorem~\ref{thm:IsoDiscontFact}, we could then construct a factor map \( \Psi \) that preserves exactly the discontinuities \( \{ \omega_{1} , \hdots , \omega_{N} \} \). However, in this setting the maximal equicontinuous factor of \( X_{\Psi( x )} \) need not be equal to the maximal equicontinuous factor of \( X_{x} \). The underlying reason is, that an odometer \( \Odo \) can only be the maximal equicontinuous factor of a Toeplitz subshift \( X_{x} \), if \( X_{x} \) is generated by a semicocycle which is invariant under no rotation in \( \Odo \), see \cite[Theorem~5.2]{DownDur_FactToepl}. That is automatically the case if there is only one discontinuity point, but may fail if there are several of them.
\end{rem}

\subsection{Necessary condition}

In Theorem~\ref{thm:IsoDiscontFact} we constructed a factor subshift with property~\ref{item:FinD}, based on the existence of finite words which appear only around a unique discontinuity point. As a sufficient condition, we used that there exists a point \( ( \omega , a ) \) in the graph closure such that \( \omega \) is locally the only discontinuity point with value \( a \). In general, this condition is sufficient, but not necessary for the existence of a factor subshift with property~\ref{item:FinD} (see Example~\ref{exmpl:OxtWithFactor} below). In the following, we weaken this condition and consider a discontinuity \( \omega \) with values \( a , b \in \Alphab \), such that locally no other discontinuity assumes both values. More formally, we call the letters \( a \neq b \in \Alphab \) an \emph{isolated value pair for \( \omega \)} (with respect to the semicocycle \( \tau_{x} \)), if there exists a neighbourhood \( [\omega]_{l} \) of \( \omega \) with \( \{ \widetilde{\omega} \in [\omega]_{l} : (\widetilde{\omega} , a) , (\widetilde{\omega} , b) \in F_{x}( \widetilde{\omega} ) \} = \{ \omega \}  \). In general, this will still not be a necessary condition for a factor with property~\ref{item:FinD} (see Example~\ref{exmpl:OxtWithFactor} again), but it becomes necessary once we additionally assume separated holes (Corollary~\ref{cor:SepHolesNecCond} below). In fact, we prove in Theorem~\ref{thm:SepHolesIsoValue} the slightly stronger statement that, with separated holes, an isolated value pair in the subshift is a necessary condition for an isolated value pair in a factor subshift. Since every point in a finite boundary is isolated, this includes factor subshifts with property~\ref{item:FinD} as a special case. The importance of separated holes for our arguments stems from the fact that for them, the sliding block code ``sees'' at most one non-periodic position. Thus, if two values for this position result in different images in the factor subshift (that is, the corresponding discontinuity is preserved), then all discontinuities with these values are preserved. Hence non-isolated value pairs have non-isolated images.

\begin{thm}
\label{thm:SepHolesIsoValue}
Let \( x \in \Alphab^{\ZZ} \) be a non-periodic Toeplitz word with separated holes, and let \( \Omega \) denote the maximal equicontinuous factor of \( X_{x} \). If there exists a factor subshift \( X_{ \Psi(x) } \) of \( X_{x} \) with the same maximal equicontinuous factor \( \Omega \) and with a point \( \omega \in B_{\Psi( X )} \) that has an isolated value pair with respect to \( \tau_{\Psi( x )} \), then there also exists a point \( \omega^{\prime} \in B_{X} \) that has an isolated value pair with respect to \( \tau_{x} \).
\end{thm}

\begin{proof}
We write \( \psi : \Alphab^{[ -J , J ]} \to \Alphab \) for the sliding block code associated to the factor map \( \Psi \colon X_{x} \to X_{ \Psi(x) } \), and \( (p_{l}) \) for a period structure of \( x \). Let \( a , b \) denote an isolated value pair for \( \omega \in B_{ \Psi( X ) } \) with respect to \( \tau_{\Psi( x )} \). Then there exist \( \widetilde{x}_{1} , \widetilde{x}_{2} \in \FactOdo_{\Psi( x )}^{-1}( \omega ) \) with \( \widetilde{x}_{1}( 0 ) =a \) and \( \widetilde{x}_{2}( 0 ) = b \). Consider now any \( x_{1} \in \Psi^{-1}( \widetilde{x}_{1} ) \) and \( x_{2} \in \Psi^{-1}( \widetilde{x}_{2} ) \). Because of \( \Psi( x_{1} )( 0 ) \neq \Psi( x_{2} )( 0 ) \), there is \( j \in [ -J , J ] \) with \( x_{1}( j ) \neq x_{2}( j ) \). We claim that \( x_{1}( j ) , x_{2}( j ) \) is an isolated value pair for \( \Rot^{j}( \omega ) \in B_{X} \) with respect to \( \tau_{x} \).

First we note that \( \FactOdo_{x}( x_{1} ) = \FactOdo_{\Psi( x )}( \Psi( x_{1} ) ) = \omega \) implies \( ( \Rot^{j}( \omega ) , x_{1}( j ) ) \in F_{x} \) by Proposition~\ref{prop:AperDx}, and similarly we obtain \(  ( \Rot^{j}( \omega ) , x_{2}( j ) ) \in F_{x} \). Assume now that our claim is false. Then there exists a sequence \( (\omega_{l})_{l} \) in \( \Odo \), such that \( \Rot^{j}( \omega_{l} ) \in [ \Rot^{j}( \omega ) ]_{l} \setminus \{ \Rot^{j}( \omega ) \} \) holds for every \( l \in \NN \), and each \( \Rot^{j}( \omega_{l} ) \) has the values \( x_{1}( j ) \) and \( x_{2}( j ) \) in \( F_{x} \). Note that for each \( \omega_{l} \) there are \( y_{1} , y_{2} \in \FactOdo_{x}^{-1}( \omega_{l} ) \) with \( y_{1}( j ) = x_{1}( j ) \) and \( y_{2}( j ) = x_{2}( j ) \), see Proposition~\ref{prop:AperDx}~\ref{item:A2}.  Since \( \FactOdo_{x}( x_{i} ) = \FactOdo_{ \Psi(x) }( \Psi( x_{i} ) ) = \omega \) and \( \FactOdo_{x}( y_{i} ) = \omega_{l} \) are in the same \( l \)-cylinder set, it follows that \( x_{i} \) and \( y_{i} \), with \( i = 1 , 2 \), agree on \( \Per( p_{l} , x_{i} ) = \Per( p_{l} , y_{i} ) \). Because of separated holes, for sufficiently large \( l \in \NN \) any two \(p_{l}\)-holes are a distance of more than \( 2J+1 \) apart. Therefore \( j \) is the only non-\( p_{l} \)-periodic position of \( x_{i} \) and \( y_{i} \) within \( [-J , J ] \), and we obtain \( x_{i}[ -J , J ] = y_{i}[ -J , J ] \). In particular this implies
\[ \Psi( y_{1} )( 0 ) = \Psi( x_{1} )( 0 ) = \widetilde{x}_{1}( 0 ) = a \quad \text{and} \quad \Psi( y_{2} )( 0 ) = b \, . \]
Together with \( \FactOdo_{\Psi( x )}(\Psi( y_{1} )) = \FactOdo_{x}( y_{1} ) = \omega_{l} = \FactOdo_{ \Psi( x ) }(\Psi( y_{2} )) \), the above yields \( (\omega_{l}, a) , (\omega_{l}, b) \in F_{ \Psi( x ) } \) for all sufficiently large \( l \). Since \( \omega_{l} \in [ \omega ]_{l} \setminus \{ \omega \} \) is arbitrarily close to \( \omega \), this contradicts that \( a , b \) is an isolated value pair for \( \omega \) with respect to \( \tau_{ \Psi( x )} \).
\end{proof}

\begin{cor}
\label{cor:SepHolesNecCond}
Let \( x \) and \( \Omega \) be as in Theorem~\ref{thm:SepHolesIsoValue}. If there exists a factor subshift \( X_{ \Psi(x) } \) of \( X_{x} \) with the same maximal equicontinuous factor \( \Omega \) and with property~\ref{item:FinD}, then there exists \( \omega \in B_{X} \) which has as isolated value pair with respect to \( \tau_{x} \).
\end{cor}

\begin{proof}
First we note that, since \( x \) is non-periodic and \( X_{ \Psi( x ) }\) has the same maximal equicontinuous factor as \( X_{x} \), also \( \Psi( x ) \) is non-periodic (see Section~\ref{subsec:Subshifts}), which implies \( B_{\Psi( X )} \neq \emptyset \). Since \( B_{\Psi( X )} \) is finite, for every \( \omega \in B_{\Psi( X )} \) there exists a neighbourhood \( [\omega]_{l} \) with \( B_{\Psi( X )} \cap [\omega]_{l} = \{ \omega \} \). In particular, \( \omega \) is the only point in \( [\omega]_{l} \) for which \( F_{\Psi(x)}( \omega ) \) is not a singleton. Hence the values in \( F_{\Psi(x)}( \omega ) \) are isolated and applying Theorem~\ref{thm:SepHolesIsoValue} finishes the proof.
\end{proof}

\subsection{Oxtoby sequences on two letters}

Many examples in this article are generalised Oxtoby subshifts on the alphabet \( \Alphab = \{ a, b \} \). Unfortunately, those satisfy neither the assumptions of Theorem~\ref{thm:IsoDiscontFact} nor that of Theorem~\ref{thm:SepHolesIsoValue}. Indeed, as we have seen in Proposition~\ref{prop:OxtoNotB}, generalised Oxtoby subshifts never have separated holes, which are required in Theorem~\ref{thm:SepHolesIsoValue}. Moreover, for two letters the notions of an isolated value pair, of an isolated point \( ( \omega , a ) \) in \( \bigcup_{ \widetilde{\omega} \in B_{X} } F_{x}( \widetilde{\omega} ) \) and of an isolated discontinuity \( \omega \) in \( B_{X} \) are all equivalent. Their existence is required in Theorem~\ref{thm:IsoDiscontFact}, but they are not present in Oxtoby subshifts, as we show next.

\begin{prop}
\label{prop:OxtoNoIso}
Let \( x \) be a generalised Oxtoby sequence with respect to a period structure \( (p_{l}) \) and let \( \Odo \) be the associated odometer. Then \( x \) has no isolated discontinuity points, that is, no \( \omega \in B_{X} \) is isolated in \( B_{X} \).
\end{prop}

\begin{proof}
For arbitrary, fixed \( \omega \in B_{X} \) and \( l \in \NN \) we will show that \( ( [ \omega ]_{l} \cap B_{X} ) \setminus \{ \omega \} \neq \emptyset \) holds. Let \( y \in \FactOdo_{x}^{-1}( \omega ) \). Our proof is based on the fact that \( \omega( l ) \) corresponds to a non-periodic position in \( x \), and because of the Oxtoby structure, there exists another non-periodic position, a fixed multiple of \( p_{l} \) away. Hence there is another boundary point in  \( [ \omega ]_{l} \). More formally, we note that Proposition~\ref{prop:AperDx}~\ref{item:A1} and the definition of \( \FactOdo_{x} \) imply
\[ \omega( l ) \in \Aper( \Shift^{-\omega( l )}( y ) ) \subseteq \Aper( p_{l} , \Shift^{-\omega( l )}( y ) ) = \Aper( p_{l} , x ) .\]
Because \( x \) is an Oxtoby sequence, \( \Aper( p_{l+1} , x ) \) equals \( \Aper( p_{l} , x ) \) on at least two intervals \( [ mp_{l} , ( m+1 )p_{l} - 1 ] \) within \( [0 , p_{l+1} - 1 ] \). Let \( s \cdot p_{l} \) with \( s \in [ 1 , \frac{ p_{l+1} }{ p_{l} }-1 ] \) denote the distance between two such intervals. Since some unfilled intervals remain completely unfilled in each step, for every \( k \geq l \) there exists \( n_{k} \in \NN \) such that \( \omega( l ) + n_{k} p_{l} \) and \( \omega( l ) + (n_{k} + s) p_{l} \) are in \( \Aper( p_{k} , x ) \). Let \( y_{1} , y_{2} \in X_{x} \) denote accumulation points of the sequences \( (\Shift^{ \omega( l ) + n_{k} p_{l}}( x ))_{k \in \NN} \) respectively \( (\Shift^{ \omega( l ) + ( n_{k} + s) p_{l}}( x ))_{k \in \NN} \) along a common subsequence of \( k \)'s. For \( i = 1, 2 \), we conclude from \( \Per( p_{l} , y_{i} ) = \Per( p_{l} , \Shift^{\omega( l )}( x ) ) \) that \( \FactOdo_{x}( y_{i} ) \in [\omega]_{l} \) holds. Moreover, for every fixed \( p_{m} \) we have for all sufficiently large \( k \) from the chosen subsequence:
\begin{align*}
\Aper( p_{m} , y_{1} ) & = \Aper( p_{m} , \Shift^{ \omega( l ) + n_{k} p_{l} }( x ) )\\
&\supseteq \Aper( p_{k} , \Shift^{ \omega( l ) + n_{k} p_{l} }( x )  ) \ni 0 .
\end{align*}
This implies \( \FactOdo_{x}( y_{1} ) \in B_{X} \), and similarly we obtain \( \FactOdo_{x}( y_{2} ) \in B_{X} \). Finally, \( y_{2} = \Shift^{ s p_{l} }( y_{1} ) \) yields \( \FactOdo_{x}( y_{2} ) \neq \FactOdo_{x}( y_{1} ) \), so that at least one of \( \FactOdo_{x}( y_{1} ) \) and \( \FactOdo_{x}( y_{2} ) \) must differ from \( \omega \).
\end{proof}

Below we discuss two examples of Oxtoby subshifts with two letters, one of them not admitting a factor subshift with property~\ref{item:FinD} over the same odometer (Example~\ref{exmpl:OxtNoFactor}) and one doing so (Example~\ref{exmpl:OxtWithFactor}). This shows that, in absence of separated holes, the necessary condition in Theorem~\ref{thm:SepHolesIsoValue} is not necessary any more. In other words: without separated holes, subshifts without isolated value pairs may or may not have factor subshifts with property~\ref{item:FinD} and the same maximal equicontinuous factor. Related to this, there are several interesting questions that we currently cannot answer:
\begin{tightitemize}
\item{What is a good criterion to distinguish these two types of behaviour, that is, when does property~\ref{item:FinD} hold for a factor subshift of a generalised Oxtoby subshift?}
\item{Generalised Oxtoby subshifts are ``relatively far'' from having separated holes, in the sense that every Toeplitz orbital \( y \) in them satisfies \( \Card{ \Aper( y ) } = \infty \) (see Proposition~\ref{prop:OxtoNotB}), as opposed to \( \Card{ \Aper( y ) } = 1 \) with separated holes. How much can we weaken the separated holes condition and still retain Theorem~\ref{thm:SepHolesIsoValue}, or how much can we restrict \( \Card{ \Aper( y ) } \) and still obtain both types of behaviour?}
\item{What are analogous statements to Theorems~\ref{thm:IsoDiscontFact} and~\ref{thm:SepHolesIsoValue} for property~\ref{item:FinOrb} instead of property~\ref{item:FinD}, either in general, in the Oxtoby setting or in some other interesting class?}
\end{tightitemize}

\begin{exmpl}
\label{exmpl:OxtNoFactor}
Let \( x \) denote the generalised Oxtoby sequence from Example~\ref{exmpl:OxtProxiAsym}. Recall that by definition, the holes are filled with words \( u^{(l)}_{i} \in \{ a , b \}^{[1,2^{l}]} \), which have a single \( b \) at position \( i \in [ 1 , 2^{l} ] \) and value \( a \) at all other positions. Let \( (p_{l}) \) denote the associated period structure of \( x \). Let \( \Psi \) be a factor map and assume that \( X_{\Psi( x )} \) has the same maximal equicontinuous factor as \( X_{x} \). In particular, \( \Psi( x ) \) is non-periodic. Let \( \psi \colon \{ a , b \}^{ [ -J , J ] } \to \{ a, b \} \) denote the sliding block code that defines \( \Psi \), and let \( l_{0} \in \NN \) be large enough such that \( [ -J , J ] \subseteq \Per( p_{l_{0}} , x ) \) holds. We fix an arbitrary \( l \geq l_{0} \) and define \( I_{m} \DefAs [m p_{l} , ( m+1 ) p_{l} - 1] \). We will show that \( \Psi( x ) \) satisfies the Oxtoby properties with respect to \( (p_{l})_{l \geq l_{0} } \), that is,
\begin{tightenumerate}
\item{\label{item:Oxto1}on each \( I_{m} \), the set \( \Aper( p_{l+1} , \Psi( x ) ) \) is empty or equal to \( \Aper( p_{l} , \Psi( x ) ) \),}
\item{\label{item:Oxto2}there are two \( m \in [ 0 , \frac{ p_{l+1} }{ p_{l} } - 1 ] \) with \( I_{m} \cap \Aper( p_{l+1} , \Psi( x ) ) \neq \emptyset\).}
\end{tightenumerate}
We warn the reader that we do not assume here that \( (p_{l})_{l} \) is a period structure of \( \Psi( x ) \); the notation \( \Aper( p_{l} , \Psi( x ) ) \) should be understood only as a statement about positions that do not have period \( p_{l} \). Before we consider the factor subshift \( X_{\Psi( x )} \) in more detail, we first prove the following main observation about \( x \):
\begin{tightenumerate}
\item[(\(\star\))]{\parbox{\linewidth}{Let \( m \in \ZZ \) be such that \( I_{m} \cap \Aper( p_{l+1} , x ) \neq \emptyset \) holds. To see all the words that appear in \( x \) at \( I_{n} \) for \( n \in \ZZ \), it suffices to consider only \( n \in m + \frac{ p_{l+1} }{ p_{l} } \ZZ \), that is: for every \( n \in \ZZ \) there exists \( \widetilde{m} \in  m + \frac{ p_{l+1} }{ p_{l} } \ZZ \) with \( x( I_{n} ) = x( I_{\widetilde{m}} ) \).}}
\end{tightenumerate}
Indeed, on the one hand \( x \) is Oxtoby, and all non-\( p_{l} \)-periodic positions in \( I_{n}  \) are therefore filled in the same step, say from \( p_{k-1} \) to \( p_{k} \) (with \( k > l \)). Thus, \( x( I_{n} \cap \Aper(p_{l} , x) ) \) is a subword of length \( 2^{l+1} \) of \( u^{(k)}_{i} \) for a suitable \( i \), that is, it is \( a^{2^{l+1}} \) or \( u^{(l+1)}_{i^{\prime}} \) for a suitable \( i^{\prime} \). On the other hand, in each \(p_{l+1}\)-interval there are two \( p_{l} \)-intervals that intersect \( \Aper( p_{l+1} , x ) \). The words \( u^{(l+2)}_{i} \) appear in the non-\( p_{l+1} \)-periodic positions of these intervals. Accordingly, for \( \widetilde{m} \in  m + \frac{ p_{l+1} }{ p_{l} } \ZZ \) all first halves or all second halves of \( u^{(l+2)}_{i} \) appear at \( I_{\widetilde{m}} \cap \Aper(p_{l} , x) \), that is, we see \( a^{2^{l+1}} \) as well as all \( u^{(l+1)}_{i} \) with \( i \in [ 1, 2^{l+1} ] \). Since additionally \( x( I_{\widetilde{m}} ) \) and \( x( I_{n} ) \) clearly are equal on their \( p_{l} \)-periodic parts, (\(\star\)) follows. As a consequence of our main observation, we obtain:
\begin{tightenumerate}
\item[(\(\star \star\))]{\parbox{\linewidth}{Let \( m \in \ZZ \) be such that \( I_{m} \cap \Aper( p_{l+1} , x ) \neq \emptyset \) holds. Then we have \( I_{m} \cap \Aper( p_{l+1 } , \Psi( x ) ) = I_{m} \cap \Aper( p_{l} , \Psi( x ) ) \).}}
\end{tightenumerate}
For a proof, first note that \( \Aper( p_{l+1} , \Psi( x ) ) \subseteq \Aper( p_{l} , \Psi( x ) ) \) is clear. For the converse inclusion, consider an arbitrary \( r \in [0 , p_{l} - 1] \cap \Aper( p_{l} , \Psi( x ) ) \). Then there exists \( n \in \ZZ \) with \( \Psi( x )( r + m p_{l} ) \neq \Psi( x )( r + n p_{l} ) \). By applying (\(\star\)) to \( m \) and \( n \), we obtain \( \widetilde{m} \in m + \frac{ p_{l+1} }{ p_{l} } \ZZ \) with \( x( I_{n} ) = x( I_{\widetilde{m}} ) \). Since \( \Psi( x )( j ) \) depends only on \( x[ j-J , j+J ] \) and since we have \( [ -J , J ] \subseteq \Per( p_{l} , x ) \), this yields \( \Psi( x )( I_{n} ) = \Psi( x )( I_{\widetilde{m}} ) \). In particular, we have
\[ \Psi( x )( r + m p_{l} ) \neq \Psi( x )( r + n p_{l} ) = \Psi( x )( r + \widetilde{m} p_{l} ) , \]
which implies \( r + m p_{l} \in \Aper( p_{l+1} , \Psi( x ) ) \) and hence~(\(\star \star\)). We can now prove the Oxtoby properties: for~\ref{item:Oxto1}, we assume that \( I_{m} \cap \Aper( p_{l+1} , \Psi( x ) ) \neq \emptyset \) holds. Because of \( [ -J , J ] \subseteq \Per( p_{l} , x ) \), this implies \( I_{m} \cap \Aper( p_{l+1} , x ) \neq \emptyset \), so~(\(\star \star\)) gives the desired result. For~\ref{item:Oxto2}, we use that \( \Psi( x ) \) is non-periodic, and hence \( I_{m} \cap \Aper( p_{l} , \Psi( x ) ) \neq \emptyset \) holds for all \( m \in \ZZ \). Since \( x \) is Oxtoby, there exist two distinct \( m_{1} , m_{2} \in [ 0 , \frac{ p_{l+1} }{ p_{l} } - 1 ] \) with \( I_{m_{i}} \cap \Aper( p_{l+1} , x ) \neq \emptyset\) for \(i = 1, 2\). By~(\(\star \star\)), we conclude that \( \Psi( x ) \) has \( p_{l+1} \)-holes in \( I_{m_{1}} \) and \( I_{m_{2}} \).

To finish the example, we will prove (as in Proposition~\ref{prop:OxtoNotB}) that the inequality \( 2^{l - l_{0} } \leq \Card{ \Aper( y ) } \) holds for every \( y \in  X_{\Psi( x )} \) with \( \Aper( y ) \neq \emptyset \). Since \( l \geq l_{0} \) was arbitrary, \( \Card{ \Aper( y ) } \in \{ 0 , \infty \} \) follows. Moreover, \( \Psi( x ) \) is non-periodic, so there exists an element with \( \Card{ \Aper( y ) } = \infty \), and hence neither~\ref{item:FinCap} nor~\ref{item:FinD} hold for \( X_{\Psi( x )} \). To prove the inequality, let \( y \) be such that \( \Aper( y ) \neq \emptyset \) holds, and let \( (q_{t})_{t} \) denote a period structure of \( \Psi( x ) \). Then there exists \( T \in \NN \) with
\[ \Aper( q_{T} , y ) \cap [ -p_{l}+1 , p_{l}-1 ] = \Aper( y ) \cap [ -p_{l}+1 , p_{l}-1 ] . \] Since the period structures \( (p_{l}) \) of \( x \) and \( (q_{t}) \) of \(  \Psi( x ) \) generate the same odometer, we can find \( L \geq l \) with \( q_{T} \mid p_{L} \) and \( S \geq T \) with \( p_{L} \mid q_{S} \). Since we can assume without loss of generality that \( 0 \in \Aper( y ) = \bigcap_{t=1}^{\infty} \Aper( q_{t} , y ) \) holds, we obtain \( 0 \in \Aper( q_{S} , y ) \subseteq \Aper( p_{L} , y ) \). By the Oxtoby properties, there exists therefore an interval of length \( p_{l} \) around zero, in which all non-\(p_{l}\)-periodic positions of \( y \) are non-\( p_{L} \)-periodic. Also by the Oxtoby properties, every interval of length \( p_{l} \) in \( y \) contains at least \( 2^{ l - l_{0} } \)-many non-\(p_{l}\)-periodic positions. We obtain
\begin{align*}
2^{l - l_{0} } &\leq \Card{ \Aper( p_{L} , y ) \cap [ -p_{l}+1 , p_{l} - 1 ] } \leq \Card{ \Aper( q_{T} , y ) \cap [ -p_{l}+1 , p_{l} - 1 ] }\\
&=  \Card{ \Aper( y ) \cap [ -p_{l}+1 , p_{l} - 1 ] } \leq  \Card{ \Aper(  y ) } .
\end{align*}
\end{exmpl}

\begin{exmpl}
\label{exmpl:OxtWithFactor}
We construct a generalised Oxtoby sequence $ \toeplw \in \{a,b\}^{\ZZ} $ with respect to the period structure \( p_{l} = 4^{l}\), with a factor subshift \( X_{\Psi( x  )} \) whose boundary \( B_{\Psi( X )} \) is a singleton. We define $\toeplw$ stepwise through hole-filling. An interval of length \( p_{l} \) consists of four  \( p_{l-1} \)-intervals, and in our construction we fill all \( p_{l-1} \)-holes in the first and in the last of them \( p_{l} \)-periodically, while all \( p_{l-1} \)-holes in the second and the third of them remain holes at level \( p_{l} \). This yields \( \Card{ \Aper( p_{l} , \toeplw ) \cap [ 0 , p_{l}-1 ] } = 2^{l} \). To define $ \toeplw $, it suffices to give for every \( l \in \NN \) the two words of length \( 2^{l-1} \) which are used to fill the \( p_{l-1} \)-holes in \( [ 0 , p_{l-1}-1 ] \) respectively \( [ 3p_{l-1} , 4p_{l-1} - 1 ] \). These words are:
\begin{tightitemize}
\item{for \( l=1 \): \( \; a \; \) and \( \; b \),}
\item{for \( l=2 \): \( \; a a \; \) and \( \; b b \),}
\item{for \( l=3 \): \( \; a a a b \; \) and \( \; b b b a \),}
\item{in general for \( l \geq 2 \): \( \; a a ( a b )^{2^{l-2}-1} \; \) and \( \; b b ( b a )^{2^{l-2}-1} \).}
\end{tightitemize}
The first steps of the process yield the following infinite words:
\begingroup
\abnormalparskip{0pt}
\setlength{\abovedisplayskip}{3pt}
\par
\scriptsize
\[
\setlength{\arraycolsep}{0.6pt}
\begin{array}{*{66}{c}}
\hdots & a & ? & ? & b & a & ? & ? & b & a & ? & ? & b & a & ? & ? & b & a & ? & ? & b & a & ? & ? & b & a & ? & ? & b & a & ? & ? & b & a & ? & ? & b & a & ? & ? & b & a & ? & ? & b & a & ? & ? & b & a & ? & ? & b & a & ? & ? & b & a & ? & ? & b & a & ? & ? & b & \hdots \\ 
\hdots & a & a & a & b & a & ? & ? & b & a & ? & ? & b & a & b & b & b & a & a & a & b & a & ? & ? & b & a & ? & ? & b & a & b & b & b & a & a & a & b & a & ? & ? & b & a & ? & ? & b & a & b & b & b & a & a & a & b & a & ? & ? & b & a & ? & ? & b & a & b & b & b & \hdots \\
\hdots & a & a & a & b & a & a & a & b & a & a & b & b & a & b & b & b & a & a & a & b & a & ? & ? & b & a & ? & ? & b & a & b & b & b & a & a & a & b & a & ? & ? & b & a & ? & ? & b & a & b & b & b & a & a & a & b & a & b & b & b & a & b & a & b & a & b & b & b & \hdots \\ 
\hdots & a & a & a & b & a & a & a & b & a & a & b & b & a & b & b & b & a & a & a & b & a & a & a & b & a & a & b & b & a & b & b & b & a & a & a & b & a & a & b & b & a & a & b & b & a & b & b & b & a & a & a & b & a & b & b & b & a & b & a & b & a & b & b & b & \hdots
\end{array}
\]
\endgroup
We note that the first non-negative \( p_{l} \)-hole is at position \( \sum_{i=0}^{l-1} 4^{i} = \frac{ 4^{l}-1 }{ 3 } \), since in every step all holes in the first interval of length \(4^{i}\) are filled and none in the second. Moreover, \( ( 4^{l} )_{l} \) is indeed a period structure: the properties~\ref{item:PerStruc:Div} and~\ref{item:PerStruc:Z} of the definition are immediate. For~\ref{defi:PerStruc:EssPeriod}, notice that every position is \( 4^{l} \)-periodic, and that its smallest period must therefore be of the form \( 2^{n} \) for some \( n \in \NN \). It is easily checked that for instance \( \frac{ 4^{l}-1 }{ 3 } \) is a \( 4^{l+1} \)-periodic position that is not \( 2 \cdot 4^{l} \)-periodic.

 Let  \( \Psi \colon X_{x} \to X_{ \Psi(x) } \) be the factor map that is defined by the sliding block code \( \psi : \Alphab^{[ -1 , 1 ]} \to \Alphab \), which is given by \(aaa \mapsto a \) and \( u \mapsto b \) for all \( u \in \Alphab^{[ -1 , 1 ]} \setminus  \{ aaa \} \). By construction, \( x \) has value \( b \) on \( 3 + 4 \ZZ \) and value \( a \) on \( 0 + 4 \ZZ \). It follows that \( \Psi( x )( j ) = a \) holds if and only if \( j \in 1 + 4 \ZZ \) and \( x( j ) = x( j+1 ) = a \) hold. Since only the first pair of holes is filled with \( aa \), we obtain that \( \Psi( x )( j ) = a \) is equivalent to \( j \in \bigcup_{l = 1}^{\infty} ( \frac{ 4^{l}-1 }{ 3 } + 4^{l+1} \ZZ ) \). In particular, for every \( l \in \NN \) there is a \( 4^{l} \)-periodic position in \( \Psi ( x) \) which is not \( 2 \cdot 4^{l-1} \)-periodic. By Proposition~\ref{prop:OdoFactIso} it now follows that \( (4^{l}) \) is a period structure of \( X_{\Psi( x ) } \) and that \( X_{x} \) and \( X_{\Psi( x ) } \) have the same odometer as their maximal equicontinuous factor. Moreover, we obtain \( \Aper( 4^{l} , \Psi( x ) ) = \frac{ 4^{l}-1 }{ 3 } + 4^{l} \ZZ \), that is, \( \Psi( x ) \) has a single non-\( p_{l} \)-periodic position per period. By Proposition~\ref{prop:BdHolesFiniteD} this implies that \( B_{ \Psi( X ) } \) is a singleton.
\end{exmpl}

\begin{rem}
\label{rem:ProxiNotAsympt}
We note that the subshift \( X_{x} \) from Example~\ref{exmpl:OxtWithFactor} has proximal orbits which are non-asymptotic (see Example~\ref{exmpl:WilliamsConstr} for another such subshift). To show that this is indeed the case, it suffices by Propositions~\ref{prop:OxtoNotB} and~\ref{prop:FiniteFiber} to check that \( \omega \DefAs ( \frac{ 4^{l}-1 }{ 3 } )_{l\geq 1} \in B_{X} \) satisfies \( \Card{ \FactOdo_{x}^{-1}( \omega ) } < \infty \). We fix \( l \in \NN \), define \( I_{m} \DefAs [m p_{l} , ( m+1 ) p_{l} - 1] \) for \( m \in \ZZ \), and consider the finite words \( x( I_{m} ) \). Clearly, all these words agree on the \( p_{l} \)-periodic positions of \( x \). Moreover, for each \( m \in \ZZ \) let \( n > l \) be minimal with \( I_{m} \subseteq \Per( p_{n} , x ) \). Since each \( I_{m} \) contains \( 2^{l} \)-many \( p_{l} \)-holes,  \( x( I_{m} \cap \Aper( p_{l} , x ) ) \) is a subword of length \( 2^{l} \) of either \( aa (ab)^{2^{n-2}-1} \) or \( bb (ba)^{2^{n-2}-1} \), starting at a multiple of \( 2^{l} \) within the word. There are only four possibilities:
\[ aa (ab)^{2^{l-1}-1} \;\; , \quad \;  (ab)^{2^{l-1}} \;\; , \quad \;  bb (ba)^{2^{l-1}-1} \;\; , \quad \;  (ba)^{2^{l-1}} . \]
It only remains to notice that for each \( y \in \FactOdo_{x}^{-1}( \omega ) \), minimality and the definition of the factor map imply \( y[ 0 , p_{l} -1 ] \in \{ x( \omega( l ) + I_{m} ) : m \in \ZZ \} \). Using this and \( [ 0 , \omega(l) - 1 ] \subseteq \Per( p_{l} , x ) \), we obtain 
\[ \Card{ \{  y[ 0 , p_{l} -1 ] \!: y \in \FactOdo_{x}^{-1}( \omega ) \} } = \Card{\{ x( \omega( l ) + I_{m} ) \!: m \in \ZZ \}} = \Card{\{ x( I_{m} ) \!: m \in \ZZ \}} \leq 4 . \]
Since \( l \in \NN \) was arbitrary, this yields \( \Card{\FactOdo_{x}^{-1}( \omega )} \leq 4\).

Alternatively, we can check directly that $\InfWord_{l} \DefAs \Shift^{k_{l}}( \toeplw )$ and $ \infWord_{l} \DefAs  \Shift^{3 \cdot 4^{l+1} + k_{l}}( \toeplw ) $, with
\[ k_{l} \DefAs 2 \cdot 4^{0} + 1 \cdot 4^{1} +  2 \cdot 4^{2} + 1 \cdot 4^{3} + \hdots + \begin{cases}
2 \cdot 4^{l} & \text{if } l \text{ is even},\\
1 \cdot 4^{l} & \text{if } l \text{ is odd},
\end{cases}\]
define two elements $ \InfWord \DefAs \lim_{l\to\infty}\InfWord_{l} $ and $ \infWord \DefAs \lim_{l\to\infty} \infWord_{l} $ which are proximal but not asymptotic. Indeed, since \( y_{l} \) and \( z_{l} \) differ only by a shift of \( 3 \cdot 4^{l+1} \), we note that \( \FactOdo_{x}( y_{l} ) \) and \( \FactOdo_{x}( z_{l} ) \) agree on their first \( l \) entries. Continuity of \( \FactOdo_{x} \) thus yields $ \FactOdo_{x}(\InfWord) = \FactOdo_{x}(\infWord) $. In particular, \( \Orb( \InfWord) \) and \( \Orb(\infWord) \) are proximal orbits by Proposition~\ref{prop:ProxiBndry}. On the other hand, \( k_{l}  \) denotes a position ``near the middle'' of the first \( 4^{l+1} \)-block; more precisely: for odd \( l \) we have 
\[ k_{l} = 2 ( 4^{0} + 4^{2} + \hdots + 4^{l-1} ) + 4 ( 4^{0} + 4^{2} + \hdots + 4^{l-1} ) = \frac{ 2 }{ 5 } 4^{ l+1 } - \frac{ 2 }{ 5 } , \]
and similarly we have \( k_{l} = \frac{ 3 }{ 5 } 4^{ l+1 } - \frac{ 2 }{ 5 }  \) for even \( l \). According to the definitions of \( y_{l} \) and \( z_{l} \) as shifts of \( x \), the origin of $\InfWord_{l}$ lies in the first \( 4^{l+1} \)-block of $\toeplw$, and the origin of $\infWord_{l} $ lies in the fourth \( 4^{l+1} \)-block. Consequently, the \( 4^{l} \)-holes around the origin of $\InfWord_{l}$ are filled by \( a a ( a b )^{2^{l}} \), while the \( 4^{l} \)-holes around the origin of $\infWord_{l}$ are filled by \( b b ( b a )^{2^{l}} \). It follows that the number of positions (left and right of the origin) on which $\InfWord_l$ and $\infWord_l$ differ, tends to infinity, so their limits $\InfWord$ and $\infWord$ are not asymptotic in either direction. By Proposition~\ref {prop:ElemToOrb}, we conclude that also the orbits \( \Orb( x ) \), \( \Orb( y ) \) are not asymptotic.
\end{rem}


\begingroup
\small
\begin{thebibliography}{BGKM02}

\bibitem[Aus88]{Ausl_MiniFlowsExtension}
J. Auslander.
\newblock {\em Minimal Flows and their Extensions}, volume 153 of {\em North-Holland Mathematics Studies}.
\newblock North-Holland Publishing Co., Amsterdam, 1988.
\newblock Notas de Matem\'{a}tica, 122. [Mathematical Notes].

\bibitem[BG13]{BaakeGrimm_Aperio}
M. Baake and U. Grimm.
\newblock {\em Aperiodic Order. {V}olume 1: {A} Mathematical Invitation}, volume 149 of {\em Encyclopedia of Mathematics and its Applications}.
\newblock Cambridge University Press, Cambridge, 2013.

\bibitem[BGKM02]{BGKM_LiYorke}
F. Blanchard, E. Glasner, S. Kolyada, and A. Maass.
\newblock On {L}i-{Y}orke pairs.
\newblock {\em Journal f{\"u}r die Reine und Angewandte Mathematik. [Crelle's Journal]}, 547:51--68, 2002.

\bibitem[BJL16]{BJL_ToeplModelSet}
M. Baake, T. J{\"a}ger, and D. Lenz.
\newblock {T}oeplitz flows and model sets.
\newblock {\em Bulletin of the London Mathematical Society}, 48(4):691--698, 2016.

\bibitem[BK90]{BuKwiat_TopolCentral}
W. Bu{\l}atek and J. Kwiatkowski.
\newblock The topological centralizers of {T}oeplitz flows and their {$Z_2$}-extensions.
\newblock {\em Publicacions Matem{\`a}tiques}, 34(1):45--65, 1990.

\bibitem[BK92]{BuKwiat_StrictErgod}
W. Bu{\l}atek and J. Kwiatkowski.
\newblock Strictly ergodic {T}oeplitz flows with positive entropies and trivial centralizers.
\newblock {\em Studia Mathematica}, 103(2):133--142, 1992.

\bibitem[CK97]{CassKar_ToeplWords}
J. Cassaigne and J. Karhum{\"a}ki.
\newblock {T}oeplitz words, generalized periodicity and periodically iterated morphisms.
\newblock {\em European Journal of Combinatorics}, 18(5):497--510, 1997.

\bibitem[CP08]{CortezPetite_GOdo}
M.~I. Cortez and S. Petite.
\newblock {$G$}-odometers and their almost one-to-one extensions.
\newblock {\em Journal of the London Mathematical Society. Second Series}, 78(1):1--20, 2008.

\bibitem[DD02]{DownDur_FactToepl}
T. Downarowicz and F. Durand.
\newblock Factors of {T}oeplitz flows and other almost {$1-1$} extensions over group rotations.
\newblock {\em Mathematica Scandinavica}, 90(1):57--72, 2002.

\bibitem[DDMP16]{DonoDuraMaaPeti-AutoGrLowCompl}
S. Donoso, F. Durand, A. Maass, and S. Petite.
\newblock On automorphism groups of low complexity subshifts.
\newblock {\em Ergodic Theory and Dynamical Systems}, 36(1):64–95, 2016.

\bibitem[DI88]{DownIwan_QuasiUnifConv}
T. Downarowicz and A. Iwanik.
\newblock Quasi-uniform convergence in compact dynamical systems.
\newblock {\em Studia Mathematica}, 89(1):11--25, 1988.

\bibitem[DKK23]{DKK-P_Minimality}
A. Dymek, S. Kasjan, and J. Ku{\l}aga{-}Przymus.
\newblock Minimality of {$\mathfrak{B}$}-free systems in number fields.
\newblock {\em Discrete and Continuous Dynamical Systems. Series A}, 43(9):3512--3548, 2023.

\bibitem[DKL95]{DownKwiatLac_CritIsomAppl}
T. Downarowicz, J. Kwiatkowski, and Y. Lacroix.
\newblock A criterion for {T}oeplitz flows to be topologically isomorphic and applications.
\newblock {\em Colloquium Mathematicum}, 68(2):219--228, 1995.

\bibitem[DL98]{DownaLacr_Alm1to1}
T. Downarowicz and Y. Lacroix.
\newblock Almost {$1$}-{$1$} extensions of {F}urstenberg-{W}eiss type and applications to {T}oeplitz flows.
\newblock {\em Studia Mathematica}, 130(2):149--170, 1998.

\bibitem[DL12]{DownLacr_ForwMProxi}
T. Downarowicz and Y. Lacroix.
\newblock Forward mean proximal pairs and zero entropy.
\newblock {\em Israel Journal of Mathematics}, 191(2):945--957, 2012.

\bibitem[Dow05]{Downa_OdomToepl}
T. Downarowicz.
\newblock Survey of odometers and {T}oeplitz flows.
\newblock In S. Kolyada, Y. Manin, and T. Ward, editors, {\em Algebraic and Topological Dynamics}, volume 385 of {\em Contemp. Math.}, pages 7--37. Amer. Math. Soc., Providence, RI, 2005.

\bibitem[Dre24]{Drwlo_StrictErgodToeplArbitEntro}
J. Drewlo.
\newblock Strictly ergodic {T}oeplitz {$\mathbb{Z}^d$}-subshifts with arbitrary entropy.
\newblock preprint, arXiv:2410.21915, 2024.

\bibitem[EG60]{EllisGottsch_HomoTransGr}
R. Ellis and W.~H. Gottschalk.
\newblock Homomorphisms of transformation groups.
\newblock {\em Transactions of the American Mathematical Society}, 94:258--271, 1960.

\bibitem[EM22]{EspinMaass_AutomGrSadic}
B. Espinoza and A. Maass.
\newblock On the automorphism group of minimal {$\mathcal{S}$}-adic subshifts of finite alphabet rank.
\newblock {\em Ergodic Theory and Dynamical Systems}, 42(9):2800--2822, 2022.

\bibitem[Fur81]{Fursten_RecurrETCombNT}
H. Furstenberg.
\newblock {\em Recurrence in Ergodic Theory and Combinatorial Number Theory}.
\newblock Princeton University Press, Princeton, NJ, 1981.
\newblock M. B. Porter Lectures.

\bibitem[GKBY06]{GKBYM_MaxPatternToepl}
N. Gjini, T. Kamae, T. Bo, and X. Yu{-}Mei.
\newblock Maximal pattern complexity for {T}oeplitz words.
\newblock {\em Ergodic Theory and Dynamical Systems}, 26(4):1073--1086, 2006.

\bibitem[GLNS22]{GLNS_LeadingSeq_Journ}
R. Grigorchuk, D. Lenz, T. Nagnibeda, and D. Sell.
\newblock Subshifts with leading sequences, uniformity of cocycles and spectra of {S}chreier graphs.
\newblock {\em Advances in Mathematics}, 407:Paper No. 108550, 34, 2022.

\bibitem[JK69]{JacobsKeane_01Toeplitz}
K. Jacobs and M. Keane.
\newblock {$0-1$}-sequences of {T}oeplitz type.
\newblock {\em Zeitschrift f{\"u}r Wahrscheinlichkeitstheorie und Verwandte Gebiete}, 13:123--131, 1969.

\bibitem[KZ02]{KamZamb_MaxPattCompl}
T. Kamae and L. Zamboni.
\newblock Maximal pattern complexity for discrete systems.
\newblock {\em Ergodic Theory and Dynamical Systems}, 22(4):1201--1214, 2002.

\bibitem[LM95]{LindMarcus_Coding}
D. Lind and B. Marcus.
\newblock {\em An Introduction to Symbolic Dynamics and Coding}.
\newblock Cambridge University Press, Cambridge, 1995.

\bibitem[LQ11]{LiuQu_Simple}
Q. Liu and Y. Qu.
\newblock Uniform convergence of {S}chr{\"o}dinger cocycles over simple {T}oeplitz subshift.
\newblock {\em Annales Henri Poincar{\'e}}, 12(1):153--172, 2011.

\bibitem[LQ12]{LiuQu_Bounded}
Q. Liu and Y. Qu.
\newblock Uniform convergence of {S}chr{\"o}dinger cocycles over bounded {T}oeplitz subshift.
\newblock {\em Annales Henri Poincar{\'e}}, 13(6):1483--1500, 2012.

\bibitem[Mar74]{Markley_CharSeq}
N.~G. Markley.
\newblock Characteristic sequences.
\newblock {\em Zeitschrift f{\"u}r Wahrscheinlichkeitstheorie und Verwandte Gebiete}, 30:321--342, 1974.

\bibitem[Moo00]{Moody_ModelSurv}
R.~V. Moody.
\newblock Model sets: A survey.
\newblock In F. Axel, F. D{\'e}noyer, and J. Gazeau, editors, {\em From Quasicrystals to More Complex Systems. {L}es {H}ouches School, {F}ebruary 23  -- {M}arch 6, 1998}, pages 145--166. Springer Berlin, Heidelberg, 2000.

\bibitem[Oxt52]{Ox_ErgSets}
J.~C. Oxtoby.
\newblock Ergodic sets.
\newblock {\em Bulletin of the American Mathematical Society}, 58:116--136, 1952.

\bibitem[Pau76]{Paul_AlmAutomFlows}
M.~E. Paul.
\newblock Construction of almost automorphic symbolic minimal flows.
\newblock {\em General Topology and its Applications}, 6(1):45--56, 1976.

\bibitem[Sel20]{Sell_SimpToepCombETDS}
D. Sell.
\newblock Combinatorics of one-dimensional simple {T}oeplitz subshifts.
\newblock {\em Ergodic Theory and Dynamical Systems}, 40(6):1673–1714, 2020.

\bibitem[Wil84]{Wil_ToepNotUniqErgod}
S. Williams.
\newblock {T}oeplitz minimal flows which are not uniquely ergodic.
\newblock {\em Zeitschrift f{\"u}r Wahrscheinlichkeitstheorie und Verwandte Gebiete}, 67(1):95--107, 1984.

\end{thebibliography}
\endgroup
\end{document}